\documentclass[9pt]{article} % For LaTeX2e
\usepackage{hyperref}
\usepackage{url}
\usepackage{smile}
\usepackage{graphicx} % more modern
\usepackage{algorithm}
\usepackage{algorithmic}
\usepackage{epstopdf}
\usepackage[margin=1.0in]{geometry}
\usepackage[export]{adjustbox}% http://ctan.org/pkg/adjustbox
\usepackage{mathtools, cuted}
\usepackage{bbm}
\usepackage{wrapfig}
\usepackage{subcaption}
\usepackage{caption}
\usepackage{enumitem}
\usepackage{tabularx}
\usepackage{smile}
\usepackage{tcolorbox}
\usepackage{enumitem}
\usepackage{mathtools, cuted}
\usepackage{caption}
\usepackage{subcaption}

\numberwithin{equation}{section}

\usepackage[nocompress]{cite}

\hypersetup{
    colorlinks=true,
    linkcolor=blue,
    anchorcolor=blue, 
    citecolor=blue,
}

\allowdisplaybreaks

\renewcommand{\frac}[2]{\tfrac{#1}{#2}}

\title{
 Sample Average Approximation for Distributionally Robust Optimization with $\phi$-divergences
    }
\author{  
    Yan Li   \thanks{Department of Industrial and Systems Engineering,
Texas A\&M University. (E-mail: \url{yan.li@tamu.edu}).}
}
\date{\vspace{-6ex}}

\begin{document}
{
\makeatletter
\addtocounter{footnote}{1} % to get dagger instead of star
\renewcommand\thefootnote{\@fnsymbol\c@footnote}%
\makeatother
\maketitle
}

\begin{abstract}
It is well known that estimating the expectation of any given bounded random variable with values in $[-B, B]$ has two basic properties: 
(1) the number of samples needed scales at the order of  $\cO(B^2/\epsilon^2)$, where $\epsilon$ is the prescribed target precision, and (2) such a sample complexity is independent of the underlying probability measure. We show that neither  of these properties can no longer hold when evaluating the worst-case expectation of the random variable, where the probability measures defining the expectation belong to a $\phi$-divergence ball centered at some nominal measure $P$. Specifically, the sample complexity and its dependence on the nominal measure can be completely characterized by the growth of the divergence function. When the divergence function $\phi$ exhibits superlinear growth,  a $P$-independent sample complexity can be obtained for sample average approximation, which depends only on the growth of $\phi$, the radius of the divergence ball, and the target precision. 
We also provide sample complexity lower bounds and demonstrate the optimality of the obtained bounds for commonly used $\phi$-divergences. 
On the other hand, when superlinear growth does not hold for $\phi$, we show that for any estimation method, evaluating the worst-case expectation has a $P$-dependent sample complexity lower bound that can be made arbitrarily large by changing $P$.
In this case, we discuss the approach of adopting the generalized definition of $\phi$-divergence, through which $P$-independent sample complexity can be recovered, and hence demonstrate the essential role of absolute continuity in affecting statistical~tractability.  
\end{abstract}

%!TEX root = ./sample_dro.tex

\section{Introduction}

Let $(\Omega, \cF)$ be a measurable space and $P$ be a probability measure on $(\Omega, \cF)$.
Let $\phi: \RR \to \RR_+ \cup \cbr{+\infty}$ be a convex lower semicontinuous function with $\phi(1)= 0$ and $\phi(x) = +\infty$ for $x < 0$.
For $\tau\geq 0$, consider  
\begin{align*}
    \mathfrak{M}_\tau(P) = \cbr{
    \zeta \in \mathfrak{D}: \int_{\Omega} \phi(\zeta(w)) d P(w) \leq \tau
    },
\end{align*}
where $\mathfrak{D}$ denotes the set of probability density functions with respect to $P$.
Let $X: \Omega \to [-B, B]$ be $\cF$-measurable for some $B \geq 0$.
The distributionally robust functional corresponding to $\phi$ is defined  as
\begin{align}\label{dro_phi}
    \cR(X) = \sup_{\zeta \in \mathfrak{D}} \int_{\Omega} X(\omega) \zeta(\omega) d P(\omega) , ~ \mathrm{s.t.}~ 
    \zeta \in \mathfrak{M}_\tau(P).
\end{align}
Going forward, we refer to $\phi$ as the divergence function associated with $\cR(X)$.

Consider the following sample average approximation (SAA) of \eqref{dro_phi}. 
That is, for any $n>0$, let $\omega_1, \ldots, \omega_n$ denote the samples drawn independently from $P$, and let $P_n$ be the corresponding empirical distribution. 
Define 
\begin{align}\label{dro_phi_empirical}
    \cR_n(X) = \sup_{\zeta \in \mathfrak{D}} \int_{\Omega} X(\omega) \zeta(\omega) d P_n(\omega) , ~ \mathrm{s.t.}~ 
    \zeta \in \mathfrak{M}_\tau(P_n).
\end{align}
We are interested in the number of samples $n$ needed for \eqref{dro_phi_empirical} to be an accurate approximation of \eqref{dro_phi}.

Clearly, when $\tau = 0$, $\cR(X)$ reduces to the expectation, and from Hoeffding's inequality one can take no more than $\cO(B^2/\epsilon^2)$  samples to estimate $\cR(X)$ via $\cR_n(X)$ up to an $\epsilon$-accuracy.
The situation seems to become unclear for $\tau > 0$. 
Sample complexity bounds have been obtained for various concrete $\phi$-divergences (e.g., total variation, Kullback-Leibler, $\chi^2$-distance, conditional value-at-risk) in the context of distributionally robust optimization \cite{duchi2021learning,levy2020large} or Markov decision processes (MDPs) \cite{yang2022toward, xu2023improved, shi2024distributionally,wang2024sample,shi2023curious}.
It is worth mentioning here that in the context of MDPs,  the obtained sample complexity bounds can depend on the nominal measure $P$ via its minimal positive probability \cite{yang2022toward, xu2023improved,shi2024distributionally}, and hence are intrinsically restricted to finite sample spaces. 
On the other hand, for a class of  Cressie-Read divergence functions, \cite{duchi2021learning} obtains the sample complexity lower and upper bounds that are independent of the nominal measure.

\vspace{0.05in}

{\bf Contributions.}
Our problem of interest is to provide a relatively complete characterization of the sample complexity for estimating $\cR(X)$ with general divergence functions.
The essential observation in this manuscript is that the sample complexity and its dependence on the nominal measure can be completely characterized by the growth of the divergence function $\phi$.
In particular, our results can be categorized based on the following dichotomy on the growth of $\phi$. 

\begin{definition}
We say that $\phi$ has sublinear growth if $\lim_{x \to \infty} \phi(x)/x < \infty$,
and superlinear growth if  $\lim_{x \to \infty} \phi(x)/x = \infty$.\footnote{
As will be clarified in Remark \ref{remark_monotone_growth},   we have that $q(x) = \phi(x)/x$ is non-decreasing and hence $\lim_{x \to \infty} \phi(x)/x$ is  well-defined.
}
\end{definition}

First, we establish the information-theoretic lower bound for estimating $\cR(X)$ if the corresponding divergence function $\phi$ grows sublinearly.
Specifically, there exists $p_{\phi, \tau}$ depending only on $\phi$ and the radius $\tau$ of the ambiguity set, such that 
for any $p \in (0, p_{\phi,\tau})$, one can construct a reference measure $P$ under which at least 
\begin{align}\label{intro_lower_sublinear}
\Omega \rbr{
\frac{1}{p}
}
\end{align}
number of samples is needed to ensure a constant estimation error with probability at least $1/2$.
%estimation error $\abs{\cR_n(X) - \cR(X)}$  is bounded away from $0$ with probability at least $1/2$, whenever $n \leq \cO(1/p)$. 
Indeed, we show that such a lower bound holds for any estimator of $\cR(X)$.
Note that the established sample complexity lower bound of $\Omega(1/p)$ holds for $p$ arbitrarily close to $0$,
thereby demonstrating the hardness of estimating the distributionally robust functional in the sublinear-growth setting.

Second, when the divergence function $\phi$ grows superlinearly, we establish a sample complexity lower bound 
of the form 
\begin{align}\label{intro_lower_suplinear}
\Omega \rbr{  {
 \frac{B^2}{\epsilon^2} +  \frac{ g_{\phi}^{-1} (\tau B / \epsilon) B }{  \epsilon } 
}
}
\end{align}
for any estimator of $\cR(X)$. 
Here $g_\phi$ denotes the growth function of $\phi$ (Definition \ref{def_growth}).
To establish the sample complexity upper bound, we utilize the dual form of \eqref{dro_phi} and represent $\cR(X)$ as the optimal value of a stochastic programming problem \cite{shapiro2017distributionally}.
Noting that the corresponding dual objective is non-Lipschitz, we then develop a novel domain restriction technique that approximates the dual problem by one with a Lipschitz objective.
We provide a precise characterization on the Lipschitz properties of the approximate objective based on the domain of the approximate problem, and determine the approximation error based on the growth of the divergence function $\phi$. 
By utilizing the approximate dual problem, we establish a sample complexity of 
\begin{align}\label{intro_upper_suplinear}
\cO \rbr{ \frac{ B \sup_{\epsilon' \in [\epsilon, B]}  g_{\phi}^{-1}(B\tau / \epsilon')
\epsilon' 
}{\epsilon^2}  +  \frac{  B g_{\phi}^{-1} (B\tau / \epsilon)}{\epsilon} }.
\end{align}
It is worth emphasizing that the obtained lower and upper bounds hold for general $\phi$-divergences with superlinear growth. 
In addition, the same analytical approach we take can also be applied to study the sample complexity for $\phi$ with sublinear growth, by adopting the generalized $\phi$-divergence in defining $\mathfrak{M}_\tau(P)$ \cite{csiszar1963informationstheoretische,ali1966general,kuhn2025distributionally}. In this case, $\mathfrak{M}_\tau (P)$ can contain probability measures that are not absolutely continuous with respect to $P$, and the resulting sample complexity becomes independent of the nominal measure $P$ (Remark \ref{remark_sublinear_generalized_phi}).

Third, to demonstrate the broad applicability of the developed sample complexity bounds, we instantiate the aforementioned results to commonly used $\phi$-divergences, and establish the optimality of the obtained rates in these cases.
In particular, we show that for a broad range of $\phi$-divergences, including essential supremum, Burg entropy, Neyman $\chi^2$, total variation, and Hellinger distance, estimating $\cR(X)$ is subject to the lower bound suggested by \eqref{intro_lower_sublinear}. 
On the other hand, for conditional value-at-risk \cite{rockafellar2000optimization}, Kullback-Leibler divergence, Jeffreys divergence, Pearson $\chi^2$, and more broadly the Cressie-Read family of divergences, we show that the lower bound \eqref{intro_lower_suplinear} and upper bound \eqref{intro_upper_suplinear} match, and hence are not improvable. 
We further highlight that all the sample complexity bounds here are obtained in a unified manner, thereby providing a systematic approach for studying distributionally robust optimization, and bypassing the case-by-case analysis carried out in the past literature for some of the aforementioned divergences \cite{duchi2021learning,levy2020large,xu2023improved, shi2024distributionally,wang2024sample}.

{\bf Organization.}
The rest of the manuscript is organized as follows. 
Section \ref{sec_sublinear} presents the sample complexity lower bound for estimating $\cR(X)$ when the divergence function $\phi$ has sublinear growth. 
Section \ref{sec_suplinear} develops sample complexity lower and upper bounds for divergence functions with superlinear growth, and discuss  its implication in obtaining an instance-dependent upper bound for $\phi$ with sublinear growth. 
We then establish sample complexities for concrete $\phi$-divergences commonly used in practice, and demonstrate the optimality of the developed results in Section \ref{sec_sublinear} and \ref{sec_suplinear}. 
Concluding remarks are made in Section \ref{sec_appendix}.

%!TEX root = ./sample_dro.tex

\section{Sample Complexity for Sublinear Growth}\label{sec_sublinear}

Throughout our discussion in this section, we make the following assumption on the divergence function $\phi$. 
\begin{assumption}\label{assump_nontrivial_phi}
There exists $\delta > 0$ such that $[1-\delta, 1+\delta] \subseteq \mathrm{int} (\mathrm{dom} (\phi))$. 
\end{assumption}

\subsection{Sample Complexity Lower Bound}

We begin with the following lower bound on $\cR(X)$ for a family of reference measures $P$ supported over $\Omega$ with two elements.

\begin{lemma}\label{lemma_risk_lb}
Consider $\Omega = \cbr{\omega_1, \omega_2}$, $\cF = 2^\Omega$, $X(\omega_1) = 1$ and $X(\omega_2) = 0$. 
Suppose $\lim_{x \to \infty} {\phi(x)}/{x} < \infty$, then for any $\tau > 0$, there exist $r_{\phi,\tau}, p_{\phi, \tau} > 0$ such that for any $P$ on $(\Omega, \cF)$ with $0 < P(\omega_1) \leq  p_{\phi, \tau} $,  
\begin{align*}
\cR(X)  \geq r_{\phi, \tau}.
\end{align*}
Here $r_{\phi,\tau}$ and $p_{\phi, \tau}$ are constants depending only on the divergence function $\phi$ and radius $\tau$.
\end{lemma}

\begin{proof}
%Consider $\Omega = \cbr{\omega_1, \omega_2}$. 
Let $p_i = P(\omega_i)$ and $\zeta_i = \zeta(\omega_i)$ for $i \in \cbr{1, 2}$. 
Recall that the constraint $\zeta \in \mathfrak{M}_\tau (P)$ requires 
\begin{align}\label{ambiguity_support_two}
\phi(\zeta_1) p_1 + \phi(\zeta_2) p_2  \leq \tau; ~ \zeta_1 p_1 + \zeta_2 p_2 = 1, ~ \zeta_1, \zeta_2 \geq 0.
\end{align}
Since $\lim_{x \to \infty} \phi(x)/ x < \infty$, 
there exist $L_\phi, G_\phi \geq 1$ such that $\phi(x) / x \leq G_\phi $ for $x \geq L_\phi$. 
In addition, given Assumption \ref{assump_nontrivial_phi}, $\phi(1) = 0$ and the fact that $\phi$ is convex, for any $\tau > 0$, there exists 
$r_{\phi, \tau} > 0$ such that 
$
\sup_{x \in [1-r_{\phi, \tau}, 1+r_{\phi, \tau}] }  \phi(x) \leq \frac{\tau}{2}.
$
Note that without loss of generality we can assume  
$r_{\phi, \tau} \leq  \min\cbr{ {\tau} / {2 G_\phi}, 1}$.

Our goal is to show the existence of $p_{\phi, \tau} > 0$ such that for any probability measure $P$ with $p_1 \in (0,  p_{\phi, \tau}]$, one can construct a corresponding $\zeta \in \mathfrak{M}_\tau(P)$ with  
$\EE_P [\zeta X]$ (and hence $\cR(X)$) bounded away from $0$.

Let $k \geq 1$ be some constant to be determined later. 
Denote $p_{\phi, \tau} =  \frac{\tau}{k G_\phi L_\phi}$, and let $\zeta_1 = \frac{\tau}{k G_\phi p_1}$. 
Note that for any  $p_1 \in (0,  p_{\phi, \tau}]$, we have 
$
\zeta_1 \geq L_\phi, 
$
and hence
$
\phi(\zeta_1) p_1 = \phi \rbr{ \frac{\tau}{k G_\phi p_1} }  \cdot \frac{k G_\phi p_1}{\tau} \cdot \frac{\tau}{k G_\phi} \leq \frac{\tau}{k},
$
where the last inequality follows from the definition of $G_\phi$ and $L_\phi$.
Given this observation and the choice of $\zeta_1$, to satisfy \eqref{ambiguity_support_two}, it suffices to have 
\begin{align*}
\frac{\tau}{k G_\phi} \leq 1,  ~ \sup_{p_2 \in [1- p_{\phi, \tau}, 1]} \phi \rbr{ \frac{1 - {\tau}/\rbr{k G_\phi}}{p_2} }  p_2 \leq \frac{(k-1) \tau}{k}.
\end{align*}
By choosing $k = \frac{\tau}{G_\phi r_{\phi, \tau}}$, given the definition of $r_{\phi, \tau}$, we have $k \geq 2$ and $\frac{\tau}{k G_\phi} = r_{\phi, \tau} \leq 1$. 
%Without loss of generality, one can assume $k \geq 2$, 
 Since $\phi$ is convex, to satisfy the above condition, it suffices to have 
\begin{align*}
%\frac{\tau}{k G_\phi} \leq 1, ~ 
\max \cbr{
\phi \rbr{ \frac{1 - r_{\phi, \tau}}{ 1 - p_{\phi, \tau}}}, \phi \rbr{ 1- r_{\phi, \tau} } 
}
\leq \frac{\tau}{2}.
\end{align*}
Note that the choice of $k$ implies $p_{\phi ,\tau} = \frac{r_{\phi, \tau}}{ L_\phi} \leq r_{\phi, \tau}$, and hence the above condition is immediately satisfied by the definition of $r_{\phi, \tau}$.
In this case, we have 
\begin{align*}
R(X) \geq \EE_P \sbr{\zeta X} = \zeta_1 p_1 = \frac{\tau}{k G_\phi p_1} \cdot p_1 = \frac{\tau}{k G_\phi} = r_{\phi, \tau},
\end{align*}
where the last inequality follows from the choice of $k$.
\end{proof}

With Lemma \ref{lemma_risk_lb}, we proceed to establish the sample complexity lower bound for sample average approximation \eqref{dro_phi_empirical}. 

\begin{theorem}
Under the same setup as in Lemma \ref{lemma_risk_lb}, let $\cR_n(X)$ be defined as in \eqref{dro_phi_empirical}.
For any $p \in (0,  p_{\phi, \tau})$,  there exists $P$ over $(\Omega, \cF)$ such that 
\begin{align*}
\abs{{\cR}_n (X) - \cR(X) } \geq {r_{\phi, \tau}} ,
\end{align*}
with probability at least $1/2$, 
if $n \leq 1/(2 p)$.
\end{theorem}

\begin{proof}
Let us adopt the same notation as in Lemma \ref{lemma_risk_lb}.
For any $p \in (0,  p_{\phi, \tau})$, define $P$ over $(\Omega, \cF)$ with 
$P(\omega_1) = p \in (0,  p_{\phi, \tau})$.
From Lemma \ref{lemma_risk_lb}, we have $\cR(X) \geq r_{\phi ,\tau}$. 
Given the definition of $X$, it is clear that $\cR_n(X) = 0$ if $P_n(\omega_1) = 0$,
and consequently 
$
\abs{{\cR}_n (X) - \cR(X) } \geq {r_{\phi, \tau}}.
$
Note that 
\begin{align*}
P^{\otimes n} \rbr{ P_n(\omega_1) = 0 } = (1-p)^n \geq 1- p n \geq 1/2,
\end{align*}
where the last inequality follows from $n \leq 1/(2p)$. 
\end{proof}

We can indeed strengthen the lower bound for SAA to the following information-theoretic lower bound for estimating $\cR(X)$. 

\begin{theorem}\label{info_lb_sublinear}
Under the same setup as in Lemma \ref{lemma_risk_lb}, for any $p \in (0,  p_{\phi, \tau})$,  and any estimator $\hat{\cR}_n(X)$ that is $\cF^{\otimes n}$-measurable,
there exists $P$ over $(\Omega, \cF)$ such that 
\begin{align*}
\abs{\hat{\cR}_n (X) - \cR(X) } \geq \frac{r_{\phi, \tau}}{2},
\end{align*}
with probability at least $1/4$, 
if $n \leq 1/(2 p)$.
\end{theorem}

\begin{proof}
Let us adopt the same notation as in Lemma \ref{lemma_risk_lb}.
Define a pair of distributions $P_0$ and $P_1$ with
$P_0(\omega_1) = 0$ and $P_1(\omega_1) = p \in (0,  p_{\phi, \tau})$.
To facilitate our discussion, let us write $\cR^P(X)$ to indicate the dependence of $\cR$ on $P$.
Then from Lemma \ref{lemma_risk_lb} and $\cR^{P_0}(X) = 0$, we have 
$
\abs{\cR^{P_1}(X) - \cR^{P_0}(X)} \geq  r_{\phi, \tau}.
$

From Le Cam's method, we have 
\begin{align*}
\inf_{\hat{R}_n (X) } \max_{P \in \cbr{P_0, P_1} } P^{\otimes n} \rbr{
\abs{\hat{R}_n (X) - R^P(X)} \geq \frac{r_{\phi, \tau}}{2}  
}
\geq 
\frac{1 - \mathrm{TV}(P_0^{\otimes n}, P_1^{\otimes n}) }{2}
= \frac{(1-p)^n}{2}
\overset{(a)}{\geq} \frac{1 - n p}{2}
\overset{(b)}{\geq} \frac{1}{4}
,
\end{align*}
where $(a)$ follows from Bernoulli's inequality and $(b)$ follows from $n \leq 1/(2 p)$. 
\end{proof}

%\begin{remark}
It could be interesting to note that the above sample complexity lower bounds suggest that when the divergence function $\phi$ exhibits sublinear growth,  the necessary number of samples for estimating $\cR_n(X)$ depends on the nominal measure $P$ and can be  made arbitrarily large.
That is, for an arbitrary number of samples $n$, there always exists $P$ over $(\Omega, \cF)$ such that the difference between true risk $\cR(X)$ and its estimate $\cR_n(X)$ is bounded below by constant $r_{\phi, \tau}$ with probability at least $1/4$. 
This should be contrasted with the risk-neutral case, where $\EE_P \sbr{X(\omega)}$ can be estimated up to $\epsilon$-accuracy with a sample complexity of $\cO(B^2/\epsilon^2)$ that is independent of  $P$.
%\end{remark}
On the other hand, when the underlying reference measure $P$ is fixed a priori, we will discuss the approach to obtain an instance-dependent (i.e., depending on $P$ and $X$) sample complexity upper bound in Section \ref{sec_suplinear}.

%!TEX root = ./sample_dro.tex

\section{Sample Complexity for Superlinear Growth}\label{sec_suplinear}

In this section, we turn our attention to $\phi$ that exhibits superlinear growth, i.e., $\lim_{x \to \infty} \phi(x) / x = +\infty$.
Let us define the corresponding growth function.

\begin{definition}\label{def_growth}
Suppose $\lim_{x \to \infty} \phi(x) / x = +\infty$.
The growth function associated with divergence $\phi$ is 
\begin{align}\label{def_growth_function}
g_\phi(x) = \begin{cases}  \inf_{x' \geq x} \frac{\phi(x')}{x'}, ~ & x \geq 1; 
\\
\infty, ~ & \text{otherwise}.
\end{cases}
\end{align}
Accordingly, we define its inverse function as 
\begin{align}\label{inverse_g_def}
g_{\phi}^{-1} (y) = \inf \cbr{x: x \geq 1, ~ g_\phi(x) \geq y}, ~ \forall y > 0.
\end{align}
\end{definition}

\begin{remark}\label{remark_monotone_growth}
Let us define the extended-value function $q(x) = \phi(x) / x$ for $x \geq 1$ and $q(x) = \infty$ otherwise. 
Clearly we have $\mathrm{dom}(q) = [1, \infty) \cap \mathrm{dom}(\phi)$.
We claim that $q(\cdot)$ is non-decreasing over $\mathrm{dom}(q)$. 
Indeed, let $x, y \in (1, \infty) \cap \mathrm{dom}(\phi)$ with $x < y$, 
there exists $\alpha \in (0,1)$ such that $x = \alpha y + (1-\alpha)$, and consequently  
\begin{align*}
q(x) = \frac{\phi(x)}{x} = \frac{\phi(\alpha y + (1-\alpha) )}{\alpha y + (1-\alpha)} \overset{(a)}{\leq} \frac{\alpha \phi(y)}{\alpha y + (1-\alpha)} \overset{(b)}{\leq} \frac{\phi(y)}{y} = q(y), 
\end{align*} 
where $(a)$ follows from the convexity of $\phi$ and $\phi(1) = 0$. 
In view of this observation, $\lim_{x \to \infty} q(x)$ is well-defined. 
In addition, since $\phi$ is convex lower semicontinuous, we have that $q(\cdot)$ is continuous in $[1,\infty) \cap \mathrm{dom}(\phi)$. 
Given the above observations, we obtain 
\begin{align}\label{g_phi_simplified}
g_\phi(x) = \begin{cases} \frac{\phi(x)}{x}, ~ & x \geq 1; 
\\
\infty, ~ & \text{otherwise}.
\end{cases}
\end{align}
If  $\phi(x), \phi(y) \in (0, \infty)$, then inequality in $(b)$ becomes strict, from which we conclude that $q(\cdot)$ (resp. $g_\phi$) is strictly increasing over $\cbr{x \in \mathrm{dom}(q): q(x) \in (0,\infty)}$ (resp. $\cbr{x \in \mathrm{dom}(g_\phi): g_\phi(x) \in (0,\infty)}$). 
In particular, we can make the following simple observation that will prove to be convenient in our ensuing discussions. 
\end{remark}

\begin{proposition}\label{prop_inverse_g}
For any $x \geq 0$, we have $g_\phi (g_\phi^{-1}(x)) \leq x$. 
In addition, for any $y > g_\phi^{-1}(x)$, we have $g_\phi(y) > x$. 
\end{proposition}
\begin{proof}
The proof follows directly from the definition of $g_\phi$, $g_\phi^{-1}$, and Remark \ref{remark_monotone_growth} above.
\end{proof}

\subsection{Sample Complexity Lower Bound}

Similar to Section \ref{sec_sublinear}, we first present a sample complexity lower bound for estimating $\cR(X)$. 
To this end, we construct the following counterpart to Lemma \ref{lemma_risk_lb} that establishes a lower bound of $\cR(X)$ for a set of nominal measures $P$.

\begin{lemma}\label{lemma_risk_lb_suplinear}
Consider $\Omega = \cbr{\omega_1, \omega_2}$, $\cF = 2^\Omega$, $X(\omega_1) = B$ and $X(\omega_2) = 0$, where $B \geq 1$. 
Suppose $\lim_{x \to \infty} {\phi(x)}/{x} = \infty$.
Then under Assumption \ref{assump_nontrivial_phi}, there exists $\epsilon_{\phi, \tau} > 0$, such that for any $\epsilon \in (0, \epsilon_{\phi, \tau})$, 
there exists $P$ over $(\Omega, \cF)$ with 
\begin{align*}
\cR(X) \geq \epsilon, ~ P(\omega_1) = \frac{\epsilon / B}{g_\phi^{-1}( \tau B / (2\epsilon))}.
\end{align*}
\end{lemma}

\begin{proof}
Let us restrict our attention to the nominal measures $P$ that have full support over $\Omega$,
and denote $p = P(\omega_1)$. 
From \eqref{dro_phi}, we have 
\begin{align}\label{two_support_supgrowth}
\cR(X) = \max \cbr{B q  : ~ q \in [0,1], ~  p \phi \rbr{ \frac{q}{p}} + (1-p) \phi \rbr{ \frac{1-q} {1-p} } \leq \tau}.
\end{align}
We proceed to show that there exists $\epsilon_{\phi, \tau} > 0$, such that for any $\epsilon \in (0, \epsilon_{\phi, \tau})$, one can construct the corresponding $P$ for which $q = \epsilon/B$ is feasible to \eqref{two_support_supgrowth}.
Indeed if this is the case, then $\cR(X) \geq Bq = \epsilon$.
To this end, it suffices to show the existence of $p(\epsilon)$ such that 
\begin{align}\label{condition_feasibility_lb_suplinear}
p(\epsilon) \phi \rbr{ \frac{\epsilon / B}{p(\epsilon)} } \leq \frac{\tau}{2}, ~ \phi \rbr{ \frac{1-\epsilon/B}{1-p(\epsilon)}} \leq \frac{\tau}{2}, ~ p(\epsilon) \in (0,1).
\end{align}
Consider taking $p(\epsilon) = \frac{\epsilon / B}{g_\phi^{-1} (\tau B/(2\epsilon))}$ in the above.
Clearly, if $p(\epsilon) \in (0,1)$,  the first and third conditions of \eqref{condition_feasibility_lb_suplinear} are satisfied.  
On the other hand, from Assumption \ref{assump_nontrivial_phi}, $\phi(1) = 0$ and the fact that $\phi$ is convex, for any $\tau > 0$, there exists 
$1 \geq r_{\phi, \tau} > 0$ such that 
$
\sup_{x \in [1-r_{\phi, \tau}, 1+r_{\phi, \tau}] }  \phi(x) \leq \frac{\tau}{2}.
$
Hence to satisfy the second condition in \eqref{condition_feasibility_lb_suplinear} it suffices to have 
$ 1 - r_{\phi, \tau} \leq \frac{1 - \epsilon/ B}{1-p(\epsilon)} \leq 1 + r_{\phi, \tau}$,
which in turn can be satisfied if $\epsilon \leq r_{\phi, \tau}$ and $p(\epsilon) \leq \frac{1}{2}$. 
 Note that $p(\epsilon)$ is non-decreasing in $\epsilon$, consequently there exists $\epsilon_{\phi, \tau} > 0$, such that for $\epsilon \in (0, \epsilon_{\phi, \tau})$, all three conditions in \eqref{condition_feasibility_lb_suplinear} can be satisfied. 
\end{proof}

With Lemma \ref{lemma_risk_lb_suplinear} in place, we can establish the following information-theoretic sample complexity lower bound when estimating $\cR(X)$.

\begin{theorem}\label{thrm_sample_lb_growth_suplinear}
Under the same setup as in Lemma \ref{lemma_risk_lb_suplinear}, for any $\epsilon \in (0,  \epsilon_{\phi, \tau})$,  and any estimator $\hat{\cR}_n(X)$ that is $\cF^{\otimes n}$-measurable,
there exists $P$ over $(\Omega, \cF)$ such that 
\begin{align*}
\abs{\hat{\cR}_n (X) - \cR(X) } \geq \frac{\epsilon}{2},
\end{align*}
with probability at least $1/4$, 
if $n \leq \frac{ g_{\phi}^{-1} (\tau B / (2\epsilon)) B }{  2\epsilon }$.
\end{theorem}

\begin{proof}
Let us adopt the same notation as in Lemma \ref{lemma_risk_lb_suplinear}.
Define a pair of distributions $P_0$ and $P_1$ with
$P_0(\omega_1) = 0$ and $P_1(\omega_1) = \frac{\epsilon / B}{g_\phi^{-1}( \tau B / (2\epsilon))}$.
Let us write $\cR^P(X)$ to indicate the dependence of $\cR$ on $P$.
Then from Lemma \ref{lemma_risk_lb_suplinear} and $\cR^{P_0}(X) = 0$, we have 
$
\abs{\cR^{P_1}(X) - \cR^{P_0}(X)} \geq  \epsilon.
$
From Le Cam's method, we have 
\begin{align*}
\inf_{\hat{R}_n (X) } \max_{P \in \cbr{P_0, P_1} } P^{\otimes n} \rbr{
\abs{\hat{R}_n (X) - R^P(X)} \geq \frac{\epsilon}{2}  
}
\geq 
\frac{1 - \mathrm{TV}(P_0^{\otimes n}, P_1^{\otimes n}) }{2}
= \frac{(1- P_1(\omega_1))^n}{2}
\overset{(a)}{\geq} \frac{1 - n P_1(\omega_1)}{2}
\overset{(b)}{\geq} \frac{1}{4}
,
\end{align*}
where $(a)$ follows from Bernoulli's inequality and $(b)$ follows from 
 $n \leq \frac{ g_{\phi}^{-1} (\tau B / (2\epsilon)) B }{  2\epsilon }$.
\end{proof}

We proceed to show that the standard $\Omega(B^2/\epsilon^2)$ sample complexity bound in the risk-neutral case also carries over to estimating $\cR(X)$.

\begin{theorem}\label{thrm_sample_quadratic_lb_suplinear}
Under the same setup as in Lemma \ref{lemma_risk_lb_suplinear}, there exists $\overline{p}_{\phi,\tau} > 0$ that only depends on $\phi$ and $\tau$, such that 
for any estimator $\hat{\cR}_n(X)$ that is $\cF^{\otimes n}$-measurable,
there exists $P$ over $(\Omega, \cF)$ for which 
 \begin{align*}
\abs{\hat{\cR}_n (X) - \cR(X) } \geq \frac{\epsilon}{8},
\end{align*}
with probability at least $1/4$, 
for any $n \leq \frac{{\overline{p}_{\phi,\tau} } B^2}{8 \epsilon^2}$ and $\epsilon \in (0, \frac{B}{4})$. 
\end{theorem}

\begin{proof}
Let us adopt the following notation.
For any distribution 
$P$ over $(\Omega, \cF)$, we write $p = P(\omega_1)$.
Conversely for any $p \in [0,1]$, we write $P$ to indicate the distribution over $(\Omega,\cF)$ with $P(\omega_1) = p$. 
%Denote $Q$ (resp. $Q'$) as the worst-case distribution for $P$ (resp. $P'$) in \eqref{dro_phi},
%and write $q = Q(\omega_1), q' = Q' (\omega_1)$.
Let us also write $\cR(p)$ in short for $\cR^P(X)$.
We first note that 
$
\lim_{p \to 0} \cR(p) = 0.
$
Indeed, suppose this is not the case, then there exists $\cbr{p_n}$ and $l > 0$ such that $\cR(p_n) \geq l B$.
Let $Q_n$ be the corresponding $({l}/{2})$-worst-case distribution for $P_n$. Then we have $q_n = Q_n(\omega_1) \geq l/2$, yet 
\begin{align*}
\lim_{n \to \infty} q_n = \lim_{n \to 0} p_n  \frac{\phi\rbr{{q_n}/{p_n}} }{\phi\rbr{{q_n}/{p_n}}  p_n / q_n}
\overset{(a)}{\leq} \tau \frac{ q_n /p_n}{ \phi\rbr{{q_n}/{p_n}}} \overset{(b)}{=} 0, 
\end{align*}
where $(a)$ follows from feasibility of $Q_n$ and $(b)$ follows from $\lim_{x \to \infty} \phi(x) / x = \infty$. Hence we obtain a contradiction. 
Consequently, there exists $0< \overline{p}_{\phi,\tau} \leq \frac{1}{2} $ such that 
\begin{align}\label{risk_ub_uniform}
\cR(p) \leq \frac{B}{2}, ~ \forall p \leq \overline{p}_{\phi,\tau}.
\end{align} 
For any $\epsilon \leq \frac{1}{4}$, let $P, P'$ be a pair of distributions with $p' = (1-\epsilon) p + \epsilon \geq p$, where $p \leq \overline{p}_{\phi,\tau}$. 
Clearly we have 
$\abs{p' - p} = p' -  p = \epsilon (1-p) \leq \epsilon$. We proceed to lower bound the difference $\abs{\cR(p) - \cR(p') }$. 
Let $q = Q(\omega_1)$, where $Q$ is the $(\epsilon B/4)$-worst-case distribution for $P$. 
Define $q' = (1-\epsilon) q + \epsilon$, then
\begin{align*}
p' \phi \rbr{\frac{q'}{p'}} + (1-p') \phi \rbr{\frac{{1- \overline q}}{1- p'}} 
& =
p' \phi \rbr{\frac{ (1-\epsilon) q + \epsilon}{(1-\epsilon) p + \epsilon}} + (1-p') \phi \rbr{\frac{ 1-q}{1-p}} \\
& = 
p'\sbr{ \phi \rbr{\frac{ (1-\epsilon) p}{(1-\epsilon) p + \epsilon} \cdot \frac{q}{p}  +  \frac{ \epsilon}{(1-\epsilon) p + \epsilon}   }  } + (1-p') \phi \rbr{\frac{ 1-q}{1-p}} \\
& \overset{(c)}{\leq} p' \cdot \frac{ (1-\epsilon) p}{(1-\epsilon) p + \epsilon} \phi \rbr{\frac{q}{p}} + (1-p') \phi \rbr{\frac{ 1-q}{1-p}} \\
& \leq p \phi \rbr{\frac{q}{p}} + (1-p) \phi \rbr{\frac{1-q}{1-p}} \leq \tau,
\end{align*}
where $(c)$ follows from $\phi$ being convex and $\phi(1) = 0$. 
The above observation then immediately implies  
$
\cR(p') \geq q'  B
$ 
and subsequently 
$
\cR(p') - \cR(p) \geq (q' - q - \frac{\epsilon}{4} ) B=[ \epsilon ( 1- q)  - \frac{\epsilon}{4} ] B \geq \frac{\epsilon B}{4},
$ 
where the first inequality follows from the definition of $q$, and the second inequality follows from \eqref{risk_ub_uniform} and that $p \leq \overline{p}_{\phi,\tau}$.
Hence by invoking Le Cam's method,
we obtain 
\begin{align*}
& \inf_{\hat{R}_n (X) } \max_{\tilde{P} \in \cbr{P, P'} } \tilde{P}^{\otimes n} \rbr{
\abs{\hat{R}_n (X) - R^{\tilde P}(X)} \geq \frac{\epsilon B}{8}  
} \\ 
 \geq & 
\frac{1 - \mathrm{TV}(P^{\otimes n}, (P')^{\otimes n}) }{2} 
 \overset{(d)}{\geq} 
\frac{1 - \sqrt{\mathrm{KL} (P^{\otimes n} \Vert (P')^{\otimes n} )/2 } }{2} 
=\frac{1 - \sqrt{n \mathrm{KL} (P \Vert P')  / 2} }{2}  
 \overset{(e)}{\geq} 
\frac{1 - \sqrt{n \mathrm{\chi}^2 (P, P') /2 } }{2}  
\overset{(f)}{\geq}  \frac{1 - \sqrt{2 n \epsilon^2/ p}}{2},
\end{align*}
where $(d)$ follows from Pinsker's inequality,  $(e)$ follows from the fact that 
$\mathrm{KL}(P,  P') \leq \chi^2(P, P')$,
and $(f)$ follows from 
$ \chi^2(P, P') = \frac{(p - p')^2}{p'(1-p')} \leq \frac{4\epsilon^2}{p}$ given that $p \leq p' = (1-\epsilon) p + \epsilon \leq \frac{3}{4}$. 
The desired claim then follows from taking $p = \overline{p}_{\phi,\tau}$. 
%from which we conclude the proof.
\end{proof}

In view of Theorem \ref{thrm_sample_lb_growth_suplinear} and \ref{thrm_sample_quadratic_lb_suplinear}, the necessary number of samples  for obtaining an $\epsilon$-accurate estimate of $\cR(X)$ is lower bounded by 
\begin{align}\label{eq_sample_lb_suplinear_max}
\Omega \rbr{  {
 \frac{B^2}{\epsilon^2} +  \frac{ g_{\phi}^{-1} (\tau B / (2\epsilon)) B }{  \epsilon } 
}
}.
\end{align}
We will later revisit this lower bound when we discuss the sample complexity  for commonly used $\phi$-divergences in Section \ref{sec_application}.

\subsection{Sample Complexity Upper Bound}\label{subsec_sample_ub}
We proceed to establish the sample complexity upper bound for estimating $\cR(X)$ that is independent of the nominal measure $P$.
Let us first recall the following representation for the dual problem to \eqref{dro_phi}.

\begin{lemma}[\cite{shapiro2017distributionally}]\label{lemma_dual_dro_phi}
We have 
\begin{align}\label{eq_phi_dro_f_form}
    \cR(X) = \inf_{\lambda > 0, \mu \in \RR}
    \cbr{f(\lambda, \mu) \coloneqq   
    \EE_P \sbr{ \lambda \tau + \mu + (\lambda \phi)^* (X - \mu) } }.
\end{align}
In addition, there exist $\lambda^* \geq 0, \mu^* \in \RR$ attaining the infimum above.
\end{lemma}

\begin{remark}\label{remark_truncation}
As suggested by Lemma \ref{lemma_dual_dro_phi}, the risk $\cR(X)$ that we aim to estimate can be represented as the optimal value of a stochastic program \eqref{eq_phi_dro_f_form}.
Consequently a natural idea for estimation is to simply consider applying the standard SAA result  \cite{shapiro2021lectures} to \eqref{eq_phi_dro_f_form}.
Unfortunately, in its vanilla form, \eqref{eq_phi_dro_f_form} oftentimes does not satisfy the basic properties required by standard SAA results.
First, the domain of the stochastic program in \eqref{eq_phi_dro_f_form} is unbounded.
Second, the objective  $f(\lambda, \mu)$ typically is not globally Lipschitz.
To some extent, these have collectively  led to numerous study in the prior literature that present remedies by utilizing some concrete forms of divergence function $\phi$ \cite{duchi2021learning,levy2020large,xu2023improved, shi2024distributionally,wang2024sample}. 
In what follows, we proceed to present some technical observations that would help us resolve the aforementioned issues in a general manner. 
\end{remark}

To handle the first issue mentioned in Remark \ref{remark_truncation}, we first establish that the minimization problem in \eqref{eq_phi_dro_f_form} indeed can be restricted to a bounded set without paying too much of a price in the optimality gap. 

\begin{lemma}\label{lemma_bounded_dual}
It holds that 
\begin{align}\label{eq_phi_dro_f_form_restrict}
    \cR(X) = \inf_{\lambda \in [0, \frac{2B}{\tau}], \mu \in [-B, B]}
    \cbr{f(\lambda, \mu) \coloneqq   
    \EE_P \sbr{ \lambda \tau + \mu + (\lambda \phi)^* (X - \mu) } }.
\end{align}
In addition, for any $\underline{\lambda} \geq 0$, we have 
\begin{align}\label{approx_restrict_lambda}
 \cR(X) \leq 
\inf_{\lambda \in [\underline{\lambda} , \frac{2B}{\tau}], \mu \in [-B, B]} f(\lambda, \mu) 
\leq \cR(X) + \underline{\lambda} \tau  .
\end{align}
\end{lemma}

\begin{proof}
For any $\lambda \geq 0$,  we have $0 \in \partial (\lambda \phi)(1)$ as  $1$ is a minimizer of $\phi$,
and hence $1 \in \partial (\lambda \phi)^*(0)$. 
Now define 
\begin{align}\label{f_def_sample}
f(\lambda, \mu; \omega) = \lambda \tau + \mu + (\lambda \phi)^*(X(\omega) - \mu),
\end{align}
 i.e.,
$f (\lambda, \mu) = \EE_{P} \sbr{f(\lambda, \mu; \omega)}$. 
For any $\lambda \geq 0$, if $\mu \geq B$, we have  $X(\omega) - \mu \leq 0$.
From monotonicity of subdifferential, we have $\partial (\lambda \phi)^* (X(\omega) - \mu) \leq 1$, and hence
\begin{align*}
d (\lambda, \mu; \omega) = 1 - \partial (\lambda \phi)^* (X(\omega) - \mu) \geq 0, ~ \forall d(\lambda, \mu; \omega) \in \partial_\mu f(\lambda, \mu; \omega).
\end{align*}
Consequently we have $\partial_\mu f(\lambda, \mu) \geq 0$ for any $\mu \geq B$. 
Similarly, one can  establish that 
for any $\mu \leq -B$, $\partial_\mu f(\lambda, \mu) \leq 0$.
Hence we obtain 
\begin{align}\label{fix_lambda_u_bound}
%\Argmin_{\mu \in \RR} f(\lambda, \mu) \subseteq [-B, B], ~ \forall \lambda \geq 0.
\inf_{\mu \in \RR} f(\lambda, \mu) = \inf_{\mu \in [-B, B]} f(\lambda, \mu).
\end{align}
In addition, since $\phi(1) = 0$, we have that for any $\lambda \geq 0$, 
$(\lambda \phi)^*(y) = \sup_{x \geq 0} xy - \phi(x) \geq y$.
This in turn implies 
\begin{align*}
f(\lambda, \mu) = \EE_P \sbr{ f(\lambda, \mu; \omega)}  \geq \EE_P \sbr{ \lambda \tau + \mu + X(\omega) - \mu} = \lambda \tau + \EE_P \sbr{X(\omega)}, ~
\forall \lambda \geq 0, \mu \in \RR
\end{align*}
From Lemma \ref{lemma_dual_dro_phi}, we also have 
$\inf_{\lambda \geq 0, \mu \in \RR} f(\lambda,  \mu) = \cR(X) \leq B$.
Hence from the above observation, for any $(\lambda^*, \mu^*)$ that is optimal to \eqref{eq_phi_dro_f_form}, we have 
\begin{align*}
B \geq f(\lambda^*, \mu^*) \geq \lambda^* \tau + E_P [X(\omega)] \geq \lambda^* \tau - B,
\end{align*}
from which we obtain $\lambda^* \leq \frac{2B}{\tau}$. 
Combining this observation with \eqref{fix_lambda_u_bound} implies \eqref{eq_phi_dro_f_form_restrict}.
Finally, note that for any $0 \leq \lambda \leq \lambda'$, we have
\begin{align*}
(\lambda \phi)^*(y) = \sup_{x} y x - \lambda \phi(x) \geq \sup_{x} y x - \lambda' \phi(x) = (\lambda' \phi)^*(y),
\end{align*}
and consequently 
\begin{align*}
f(\lambda, \mu) = \EE_P \sbr{\lambda \tau + \mu + (\lambda \phi)^* (X(\omega) - \mu)} 
\geq 
 \EE_P \sbr{\lambda \tau + \mu + (\lambda' \phi)^* (X(\omega) - \mu)} 
 = f(\lambda', \mu) + (\lambda - \lambda') \tau,
\end{align*}
from which we obtain \eqref{approx_restrict_lambda}.
\end{proof}

We now proceed to discuss the strategy for handling the potential non-Lipschitz property associated with $f(\lambda, \mu)$ in \eqref{eq_phi_dro_f_form_restrict}. 
In particular, 
in view of \eqref{eq_phi_dro_f_form_restrict} and the fact that $\norm{X}_\infty \leq B$, 
it is evident that the we can restrict our attention to $\partial_\lambda (\lambda \phi)^*(y)$ 
and $\partial_y (\lambda \phi)^*(y)$ 
for $\abs{y} \leq 2B$. 
Since
$
(\lambda \phi)^*(y) = \sup_{x} yx - \lambda \phi(x) 
$,
 $\phi$ grows superlinearly and is lower semicontinuous, 
for any $\lambda > 0$, one can readily see that the supremum is attainable. 
Consequently, from Danskin's theorem, we obtain 
\begin{align}
\partial_y (\lambda \phi)^*(y) &  = \mathrm{Conv} \rbr{\Argmax_x yx - \lambda \phi(x) } \overset{(a)}{=} \Argmax_x yx - \lambda \phi(x),  
\label{subgrad_u_dual}
\\
\partial_\lambda (\lambda \phi)^*(y) & = \mathrm{Conv} \rbr{\cbr{\phi(x^*): x^* \in \Argmax_{x} yx - \lambda \phi(x) }} 
\overset{(b)}{=} \cbr{\phi(x^*): x^* \in \Argmax_{x} yx - \lambda \phi(x) },
\label{subgrad_lambda_dual}
\end{align}
where $(a)$ follows from that $\Argmax_x yx - \lambda \phi(x)$ is convex,
and $(b)$ follows from  the fact that 
$
\{\phi(x^*): x^* \in \Argmax_{x} yx - \lambda \phi(x) \}
$ is convex
since
$
  \phi(\overline{x}^*)  = y (\overline{x}^* - x^*) / \lambda + \phi(x^*) 
$ 
for any $x^*, \overline{x}^* \in \Argmax_{x} yx - \lambda \phi(x).$
In view of \eqref{subgrad_u_dual} and \eqref{subgrad_lambda_dual}, the Lipschitz property associated with \eqref{eq_phi_dro_f_form_restrict} over the restricted domain clearly depends on the magnitude of the corresponding maximizer for defining the dual problem.

Given the above observation,   
our basic idea going forward is to replace the divergence function $\phi$ in \eqref{dro_phi}  by its truncated variant $\phi_L$, where 
\begin{align*}
\phi_L(\cdot) = \phi (\cdot) + \mathbbm{1}_{[0,L]}(\cdot),
\end{align*}
and $\mathbbm{1}_{[0,L]}(\cdot)$ denotes the set indicator function of $[0,L]$. 
Note that $\phi_L$ is also a valid divergence function whenever $L \geq 1$.
Let us define the truncated risk 
\begin{align}\label{phi_d_truncate}
        \cR_L(X) = \sup_{\zeta \in \mathfrak{D}} \int_{\Omega} X(\omega) \zeta(\omega) d P(\omega) , ~ \mathrm{s.t.}~ 
         \int_{\Omega} \phi_L(\zeta(w)) d P(w) \leq \tau.
\end{align}
With the truncation in place, it can be readily seen that $\partial_y (\lambda \phi_L)^*(y)$ and $\partial_\lambda (\lambda \phi_L)^*(y)$ are both bounded sets, 
and consequently \eqref{eq_phi_dro_f_form_restrict} with $\phi$ replaced by $\phi_L$ becomes a standard stochastic program with Lipschitz continuous objective. 
We formalize this in the ensuing Lemma \ref{lemma_lipschitz}.\footnote{
It could be worth mentioning that Lemma \ref{lemma_lipschitz} holds without requiring $\phi$ to have superlinear growth. 
}

\begin{lemma}\label{lemma_lipschitz}
For any divergence function $\phi$, any $L \geq 1$ and $\underline{\lambda} > 0$,
we have that
\begin{align}\label{def_f_truncate}
f_L(\lambda, \mu; \omega) \coloneqq  \lambda \tau + \mu + (\lambda \phi_L)^*(X(\omega) - \mu)
\end{align}
is Lipschitz continuous w.r.t $\mu$ (resp. $\lambda$) with modulus $L+1$ (resp. $\frac{2LB}{\underline{\lambda}} + \tau$) 
over domain 
$ [\underline{\lambda},  \frac{2B}{\tau}] \times [-B, B]$. 
In addition, we have 
\begin{align}\label{dual_stoch_obj_bound}
\abs{f_L(\lambda, \mu; \omega) } \leq (3 + 2L) B , ~ \forall 
\mu  \in [-B, B], ~ \lambda \in [\underline{\lambda},  \frac{2B}{\tau}].
\end{align}
\end{lemma}

\begin{proof}
For any $x^* \in \Argmax yx - \lambda \phi(x)$ (note that such $x^*$ exists as $\phi_L$ is lower semicontinuous), we have $0 \leq x^* \leq L$. 
Combining this with \eqref{subgrad_u_dual}, we have that 
$f_L(\lambda, \cdot; \omega)$ is Lipschitz continuous with modulus $L+1$. 

On the other hand, for any $\abs{y} \leq 2B$, since $\phi_L(1) = 0$, from the definition of $x^*$, we obtain  
$y x^* - \lambda \phi_L(x^*) \geq  y $.
This in turn implies 
$
\phi_L(x^*) \leq \frac{y x^* - y }{\lambda} \leq \frac{2 LB }{\underline{\lambda}}.
$ 
Combining this with \eqref{subgrad_lambda_dual}, 
we obtain that $f_L(\cdot, \mu; \omega)$ is Lipschitz continuous with modulus $\frac{2LB}{\underline{\lambda}} + \tau $. 
In addition, for any $\abs{y} \leq 2B$, we have  $(\lambda \phi_L)^*(y) \geq y \geq -2B$, and 
$(\lambda \phi_L)^*(y) \leq y x^* \leq 2B L$.
Consequently, it holds that
\begin{align*}
\abs{f_L(\lambda, \mu; \omega) } \leq \frac{2B}{\tau} \cdot \tau + B + 2B L   \leq (3 + 2L) B,
~~~ 
\forall \mu  \in [-B, B], ~ \lambda \in [\underline{\lambda}, \frac{2B}{\tau}],
\end{align*}
from which we conclude the proof.
\end{proof}

In view of   Lemma \ref{lemma_bounded_dual} and \ref{lemma_lipschitz},  one can then consider applying the standard SAA result to estimate $\cR_L(X)$. 
Hence conceptually the only remaining task is to control the difference between $\cR_L(X)$ and $\cR(X)$, which we discuss below.
In particular, we will establish that the approximation error introduced by the truncation scales inversely with respect to the growth of divergence function $\phi$. 

\begin{lemma}\label{lemma_risk_via_truncation}
Suppose $\lim_{x \to \infty} \phi(x) / x = +\infty$.
For any $L \geq 1$, it holds that 
\begin{align}\label{ineq_truncate_tail_via_growth}
\cR(X) - \frac{B\tau}{g_\phi(L)} \leq \cR_L(X) \leq \cR(X).
\end{align}
\end{lemma}

\begin{proof}
We proceed to first show that for any $\zeta$ feasible to \eqref{dro_phi}, there exists $\tilde{\zeta}$ feasible to \eqref{phi_d_truncate} such that 
    \begin{align}\label{approx_risk_truncate}
     \EE_P [ \tilde{\zeta} X ]
    \geq \EE_P \sbr{\zeta X} - 
    B \EE_P \sbr{\zeta - L}_+
    \end{align}
for any $L \geq 1 $.
Without loss of generality, let us assume that $\zeta > L$ on for some set $A$ with $P(A) > 0$.     
Otherwise one can take $\tilde{\zeta} = \min \cbr{\zeta, L}$ and \eqref{approx_risk_truncate} holds trivially. 
For any $v \in [0,1]$, let us consider the truncation of $\zeta$, denoted by $\zeta^v$, defined as
$
\zeta^v(\omega) = \zeta(\omega) \mathbbm{1}_{\cbr{\zeta(\omega) \in [v, L]}}
+ L \mathbbm{1}_{\cbr{\zeta (\omega) > L }}
+ v \mathbbm{1}_{\cbr{\zeta(\omega) < v}}
$.
We have 
\begin{align*}
\int \zeta^0 dP(\omega) = \int \min \cbr{\zeta, L} dP(\omega) < 1; ~
\int \zeta^1 dP(\omega) = \int 1 d P (\omega)  + \int_{\omega: \zeta(\omega) > L} (L-1) dP(\omega) \geq 1. 
\end{align*}
Since $\psi(v) \coloneqq \int \zeta^v dP(\omega)$ is continuous in $v$,  there exists $v^* \in (0,1]$ such that 
$ \zeta^{v^*}$ satisfies 
\begin{align*}
\int  \zeta^{v^*} dP(\omega) = 1, ~  \zeta^{v^*} \geq 0, ~ \int \phi( \zeta^{v^*}) dP(\omega) \leq \tau,
\end{align*}
where the last inequality follows from the fact that $1$ is the minimizer of $\phi(x)$, $\phi$ is convex, and 
that $\abs{\zeta^{v^*}(\omega) - 1} \leq \abs{\zeta(\omega) - 1}$ for any $\omega \in \Omega$ (which implies $\phi(\zeta^{v^*}(\omega) \leq \phi(\zeta(\omega))$). 
Now clearly, we have 
\begin{align*}
\EE_P \sbr{ ( \zeta^{v^*} - \zeta )X } 
& = \int_{\omega: \zeta(\omega) > L} (L - \zeta(\omega) )X (\omega) dP(\omega)
+ \int_{\omega: \zeta(\omega) < v^*} (v^* - \zeta(\omega)) X  dP(\omega)  \\
& \geq \int_{\omega: \zeta(\omega) > L} (L - \zeta(\omega) )X (\omega) dP(\omega) \\ 
& \geq - B \int [\zeta - L]_+ dP.
\end{align*}
Hence \eqref{approx_risk_truncate} holds with $\tilde{\zeta} = \zeta^{v^*}$. 
Now let $g_\phi$ be the growth function defined as in \eqref{def_growth_function}, we obtain 
\begin{align}\label{prob_mass_truncate}
 \EE_P \sbr{\zeta - L}_+ \leq \int_{\zeta \geq L}  \zeta(\omega) dP(\omega) =  \int_{\zeta \geq L}  \frac{ \phi(\zeta(\omega))}{\phi(\zeta(\omega)) / \zeta(\omega)}  dP(\omega)
 \overset{(a)}{\leq} \int_{\zeta \geq L} \frac{\phi(\zeta)}{g_\phi(L)} dP(\omega) \overset{(b)}{\leq} \frac{\tau}{g_\phi(L)},
\end{align}
where $(a)$ follows from the definition of $g_\phi$ and $L \geq 1$, and $(b)$ follows from the fact that $\zeta$ is feasible to \eqref{dro_phi}.
Consider $\zeta$ being any feasible solution to \eqref{dro_phi}. From  \eqref{approx_risk_truncate} and \eqref{prob_mass_truncate},  we obtain
\begin{align*}
\cR_L(X) \geq \EE_P [{\tilde{\zeta} X} ] \geq \EE_P \sbr{\zeta X} - B \EE_P \sbr{\zeta - L}_+ \geq \EE_P \sbr{\zeta X}  - \frac{B \tau}{g_\phi(L)}.
\end{align*}
We can then conclude the proof by taking supremum over $\zeta$.
\end{proof}

We are now ready to present the sample complexity result for 
 estimating $\cR(X)$ via sample average approximation.
To facilitate our discussion, let us define $\cR_{n, L}$ as the SAA estimator defined as in \eqref{dro_phi_empirical}, with $\phi$ replaced by $\phi_L$,
and 
\begin{align}\label{eq_def_optimal_dual}
(\lambda^*, \mu^*) \in \Argmin_{\lambda \in [\underline{\lambda}, \frac{2B}{\tau}], \mu \in [-B, B]} \EE_P \sbr{f_L(\lambda, \mu; \omega)}, ~~~ 
(\hat{\lambda}^*, \hat{\mu}^*) \in \Argmin_{\lambda \in [\underline{\lambda}, \frac{2B}{\tau}], \mu \in [-B, B]} \EE_{P_n} \sbr{f_L(\lambda, \mu; \omega)}.
\end{align}
Note that $(\lambda^*, \mu^*)$ and $(\hat{\lambda}^*, \hat{\mu}^*)$ are well-defined give Lemma \ref{lemma_lipschitz}.

\begin{theorem}\label{thrm_sample_p_independent_truncation}
Suppose $\lim_{x \to \infty} \phi(x) / x = \infty$. 
For any $\epsilon > 0$, let 
 \begin{align}\label{eq_lip_m_constants}
 L(\epsilon)  =  g_\phi^{-1}\rbr{\frac{4 B\tau }{ \epsilon}} , ~~ \underline{\lambda}(\epsilon) =\frac{\epsilon}{4\tau}, ~ ~
 \overline{L} (\epsilon) = {L (\epsilon) +1 +  \frac{2L(\epsilon)B}{\underline{\lambda}(\epsilon)}  + \tau},  ~~ 
 M(\epsilon) = 
  (3 + 2L(\epsilon) ) B  .
 \end{align}
 Then for any $\delta \in (0,1)$, we have $\abs{\cR(X) - \cR_n(X)} \leq \epsilon$ with probability at least $1-\delta$
 if 
 \begin{align*}
 n \geq 
 \frac{
 32 M(\epsilon)^2 \log(256 B^2 \overline{L}^2 (\epsilon)/ (\tau \epsilon^2 \delta)) 
 }{\epsilon^2}.
 \end{align*}
\end{theorem}

\begin{proof}
From Lemma \ref{lemma_bounded_dual}, 
we have 
\begin{align}
\cR_L(X)   & \leq   \EE_P \sbr{f_L (\lambda^*, \mu^*; \omega)} \leq   \cR_L(X) + \underline{\lambda} \tau , \label{risk_via_emp_restrict} \\
   \cR_{n,L}(X)  & \leq  \EE_{P_n} \big[{f_L (\hat{\lambda}^*, \hat{\mu}^*; \omega)} \big] \leq \cR_{n,L}(X) +  \underline{\lambda}  \tau . \label{risk_via_emp_restrict_truncate}
\end{align}
For any $\varepsilon > 0$, let $\cN_1$ (resp. $\cN_2$) be an $\varepsilon$-net of $[\underline{\lambda}, \frac{2B}{\tau}]$ (resp. $[-B, B]$). 
Correspondingly, for any $\lambda \in  [\underline{\lambda}, \frac{2B}{\tau}]$ and $\mu \in [-B, B]$, let 
$\lambda_\varepsilon$ and $\mu_\varepsilon$ be their corresponding closest points in $\cN_1$ and $\cN_2$ respectively. 
From Hoeffding's inequality, \eqref{dual_stoch_obj_bound} in Lemma \ref{lemma_lipschitz}, and the union bound, we have that for any $\delta \in (0,1)$, with probability $1-\delta$,
\begin{align*}
\abs{
\EE_{P} \sbr{f_L (\lambda_\varepsilon, \mu_\varepsilon; \omega)}  - \EE_{P_n} \sbr{f_L (\lambda_\varepsilon, \mu_\varepsilon; \omega)}
}
\leq 
(3+2L)B \sqrt{\frac{2 \log(4 B^2/(\tau \varepsilon^2 \delta))}{n}}, 
~ \forall \lambda \in [\underline{\lambda}, \frac{2B}{\tau}], \mu \in [-B, B].
\end{align*}
Consequently, from Lemma \ref{lemma_lipschitz},  we have
\begin{align*}
\abs{
\EE_{P} \sbr{f_L (\lambda, \mu; \omega)}  - \EE_{P_n} \sbr{f_L (\lambda, \mu; \omega)}
}
\leq 
(3+2L)B \sqrt{\frac{2 \log(4B^2/(\tau \varepsilon^2 \delta))}{n}} 
+ 2  \rbr{L+1 +  \frac{2LB}{\underline{\lambda}} + \tau} \varepsilon
\end{align*}  
for any $\lambda \in [\underline{\lambda}, \frac{2B}{\tau}], \mu \in [-B, B]$ 
with probability $1-\delta$.
Combining the above observation with \eqref{risk_via_emp_restrict} and \eqref{risk_via_emp_restrict_truncate}, we obtain that 
\begin{align}\label{ineq_contration_risk_truncated_raw}
\abs{\cR_L(X) - \cR_{n,L}(X)} \leq 
\underline{\lambda} \tau + 
(3+2L)B \sqrt{\frac{2 \log(4B^2/(\tau \varepsilon^2 \delta))}{n}} 
+ 2  \rbr{L+1 +  \frac{2LB}{\underline{\lambda}} + \tau} \varepsilon
\end{align}
holds with probability $1-\delta$,
which together with Lemma \ref{lemma_risk_via_truncation}, implies 
\begin{align*}
\abs{\cR(X) - \cR_{n}(X)} \leq 
\underline{\lambda} \tau + 
(3+2L)B \sqrt{\frac{2 \log(4B^2/(\tau \varepsilon^2 \delta))}{n}} 
+ 2  \rbr{L+1 +  \frac{2LB}{\underline{\lambda}} + \tau} \varepsilon
+ \frac{B \tau}{g_\phi(L)}.
\end{align*}
The desired claim then follows by substituting the choice of 
$\underline{\lambda} = \underline{\lambda}(\epsilon)$,
$L = L(\epsilon)$, and $\varepsilon = \frac{\epsilon}{8 \overline{L}(\epsilon)}$ into the above relation.
\end{proof}

It should be noted that Theorem \ref{thrm_sample_p_independent_truncation}  can also be applied to divergence functions $\phi$ with sublinear growth, provided that the estimation accuracy $\epsilon$ is not too small such that $L(\epsilon)$ is well-defined. 
As our last result in this section, one can also consider strengthening  Theorem \ref{thrm_sample_p_independent_truncation} to the following instance-dependent sample complexity bound when estimating $\cR(X)$. 
It could be worth mentioning here that the obtained instance-dependent bound holds for any divergence function $\phi$ and does not require its superlinear growth.

\begin{theorem}\label{thrm_instance_dependent}
 For $\epsilon > 0$, 
 let $\zeta^*_{\epsilon}$ be an $(\epsilon/8)$-optimal solution of \eqref{dro_phi}. 
 Define 
 \begin{align}\label{def_L_instance}
 L(\epsilon) =\max \cbr{g_\phi^{-1}\rbr{\frac{4 B\tau }{ \epsilon}},  \inf \cbr{L: \EE_P [\zeta^*_\epsilon - L]_+ \leq \frac{\epsilon}{8B}} }.
 \end{align}
%  and $\underline{\lambda} (\epsilon) = \epsilon / (8\tau)$. 
In addition, let $\underline{\lambda} (\epsilon), \overline{L}(\epsilon)$ and $M(\epsilon)$ be defined as in \eqref{eq_lip_m_constants}.
 Then for any $\delta \in (0,1)$, we have $\abs{\cR(X) - \cR_{n,L}(X)} \leq \epsilon$ with probability at least $1-\delta$
 if 
 \begin{align*}
 n \geq 
 \frac{
 32 M(\epsilon)^2 \log(256 B^2 \overline{L}^2 (\epsilon)/ (\tau \epsilon^2 \delta)) 
 }{\epsilon^2}.
 \end{align*}
\end{theorem} 

\begin{proof}
We first note that $L(\epsilon)$ is well-defined. Indeed, since $\EE_P \sbr{\zeta^*_\epsilon - L}_+ \leq \EE_P \sbr{\zeta^*_\epsilon} = 1$, the dominated convergence theorem implies that 
$\lim_{L \to \infty} \EE_P \sbr{\zeta^*_\epsilon - L}_+  = 0$. 
Consequently from \eqref{approx_risk_truncate}, the definition of $\zeta_\epsilon^*$, and the definition of $L(\epsilon)$ in \eqref{def_L_instance},
we obtain
\begin{align*}
%\label{err_truncation_instance}
 \cR(X) - \frac{\epsilon}{4} \leq  \EE_P [\zeta^*_\epsilon X ]  \leq \cR_L(X) \leq \cR(X) , 
\end{align*}
where $L = L(\epsilon)$. 
Note that \eqref{ineq_contration_risk_truncated_raw} still holds. 
Combining the above relation with \eqref{ineq_contration_risk_truncated_raw} then yields 
\begin{align*}
&  \cR(X)  - \big({
 \underline{\lambda} \tau + 
(3+2L)B \sqrt{\frac{2 \log(4B^2/(\tau \varepsilon^2 \delta))}{n}} 
+ 2  \rbr{L+1 +  \frac{2LB}{\underline{\lambda}} + \tau} \varepsilon
} \big)- \frac{\epsilon}{4}  \\
   \leq  & 
 \cR_{n,L}(X)  \\
 \leq  &  \cR(X) + \underline{\lambda} \tau + 
(3+2L)B \sqrt{\frac{2 \log(4B^2/(\tau \varepsilon^2 \delta))}{n}} 
+ 2  \rbr{L+1 +  \frac{2LB}{\underline{\lambda}} + \tau} \varepsilon
+ \frac{B\tau}{g_\phi(L)} .
\end{align*}
The desired claim then follows by substituting the choice of 
$\underline{\lambda} = \underline{\lambda}(\epsilon)$,
$L = L(\epsilon)$ into the above relation, 
while taking 
$\varepsilon = \frac{\epsilon}{8 \overline{L}(\epsilon)}$. 
\end{proof}

In view of Theorem \ref{thrm_sample_p_independent_truncation}, the sample complexity for obtaining an $\epsilon$-accurate $\cR_n(X)$ as an estimate of $\cR(X)$ is bounded by 
$\cO(M(\epsilon)^2/\epsilon^2) = \cO(B^2 L(\epsilon)^2 / \epsilon^2)$. 
It turns out that this can be further improved by refining our argument in Theorem \ref{thrm_sample_p_independent_truncation}, which we proceed to discuss in the next section.

%\yan{
%Should mention that truncation is not need if we do not want instance-dependent sample complexity. 
%}
%

\subsection{Improved Sample Complexity Upper Bound}\label{subsec_improved_sample_complexity}

In this section, we discuss some improvements to the obtained sample complexities in Section \ref{subsec_sample_ub}. 
Note that since $\cR(X+ c) = \cR(X) + c$, without loss of generality we can assume $0 \leq X \leq B$. 
Given this and the fact that  $\phi^*(y) \geq y$,   we have
\begin{align}\label{f_nonnegative}
f_L(\lambda, \mu; \omega) =  \lambda \tau + \mu + (\lambda \phi_L)^*(X(\omega) - \mu)
\geq  \lambda \tau + \mu + X(\omega) - \mu \geq 0,
\end{align}
for any $L \geq 0$ and $(\lambda, \mu) \in  [\underline{\lambda}, \frac{2B}{\tau}] \times [-B, B]$. 

We begin by noting that Theorem \ref{thrm_sample_p_independent_truncation} only uses the boundedness of $f_L$ over the domain of our interest $(\lambda, \mu) \in [\underline{\lambda}, \frac{2B}{\tau}] \times [-B, B]$.
On the other hand, it follows from \eqref{risk_via_emp_restrict} that at the optimal dual variable $(\lambda^*, \mu^*)$, we have $\EE_P \sbr{f_L(\lambda^*, \mu^*; \omega) }  \leq B + \underline{\lambda} \tau$. 
Given this and \eqref{f_nonnegative}, one can immediately control $\EE_P \sbr{f_L^2(\lambda^*, \mu^*; \omega)}$ by  $\cO(B M (\epsilon) )$ instead of $\cO(M^2(\epsilon))$.   
This suggests that one can potentially strengthen Theorem \ref{thrm_sample_p_independent_truncation} by a direct application of the Bernstein  inequality.

\begin{lemma}\label{lemma_two_side_bernstein}
For any random variable $Z$ over $(\Omega, \cF, P)$ with $0 \leq Z \leq M$. 
Let $P_n$ denote the empirical distribution constructed from $n$ independent samples following distribution $P$. 
For any $\delta \in (0,1)$, with probability at least $1-\delta$, we have
\begin{align}
{ \EE_{P_n} \sbr{Z}  - \EE_P \sbr{Z} } & \leq  4  \bigg({
\sqrt{\frac{  { \EE_P [{Z^2}]} \log(3/\delta)}{n}} + \frac{M \log(3/\delta)}{n} 
} \bigg),  \label{eq_bernstein_1} \\
{  \EE_P \sbr{Z} - \EE_{P_n} \sbr{Z}   } & \leq  4 \bigg({
\sqrt{\frac{  { \EE_{P_n} [{Z^2}]} \log(3/\delta)}{n}} + \frac{M \log(3/\delta)}{n} 
} \bigg). \label{eq_bernstein_2}
\end{align} 
\end{lemma}

\begin{proof}
We start by noting that \eqref{eq_bernstein_1} follows directly from the standard Bernstein inequality. 
On the other hand, since $0 \leq Z \leq M$, applying the Bernstein inequality to $Z^2$ yields 
\begin{align*}
\EE_P \sbr{Z^2} - \EE_{P_n} \sbr{Z^2} \leq M \sqrt{\frac{2 \EE_P[Z^2] \log(1/\delta) }{n} }  + \frac{M^2 \log(1/\delta)}{3n}
\end{align*}
with probability $1-\delta$. 
Solving for $\sqrt{\EE_P [Z^2]}$ in the above relation, we have 
\begin{align}\label{berstein_second_moment}
\sqrt{\EE_P [Z^2]} \leq  \sqrt{\EE_{P_n} [Z^2]} +  M (\sqrt{2} + 1/\sqrt{3}) \sqrt{\frac{\log(1/\delta)}{n} } .
\end{align}
Now applying the Bernstein inequality to $Z$, we obtain 
\begin{align*}
\EE_P [Z] - \EE_{P_n} [Z] \leq \sqrt{\frac{2  \log(1/\delta)}{n}} \cdot \EE_P[Z^2]+ \frac{M \log(1/\delta)}{3n}.
\end{align*}
Then \eqref{eq_bernstein_2} follows after substituting \eqref{berstein_second_moment} into the above relation and applying the union bound. 
\end{proof}

With Lemma \ref{lemma_two_side_bernstein} in place, one can then improve Theorem \ref{thrm_sample_p_independent_truncation} as follows.

\begin{theorem}\label{thrm_sample_bernstein}
Suppose $\lim_{x \to \infty} \phi(x) / x = \infty$. 
For any $0 < \epsilon \leq B $, let 
$\underline{\lambda}(\epsilon), L(\epsilon)$, 
 $\overline{L}(\epsilon)$ and $M(\epsilon)$ be defined as in \eqref{eq_lip_m_constants}.
Then 
for any $\delta \in (0,1)$  
we have $\abs{\cR(X) - \cR_n(X)} \leq \epsilon$ with probability at least $1-\delta$
 if 
\begin{align*}
n \geq 
\frac{2048 B M(\epsilon)  \log (768 B^2 \overline{L}^2(\epsilon)/ (\epsilon^2 \tau \delta))}{\epsilon^2} 
+ 
\frac{32 M(\epsilon)  \log (768 B^2 \overline{L}^2(\epsilon)/ (\epsilon^2 \tau \delta))}{\epsilon} .
\end{align*}
\end{theorem}

\begin{proof}
We begin by noting that  
\begin{align}
%\label{pop_emp_min_loss_diff_2}
& \EE_{P_n} \big[{f_L (\hat{\lambda}^*, \hat{\mu}^*; \omega)}\big]
- 
\EE_{P} \sbr{f_L ({\lambda}^*, {\mu}^*; \omega)} \nonumber \\
\leq & 
\EE_{P_n} \sbr{f_L ({\lambda}^*, {\mu}^*; \omega)}
- 
\EE_{P} \sbr{f_L ({\lambda}^*, {\mu}^*; \omega)} 
\nonumber \\
\overset{(a)}{\leq} & 
 4 \bigg({
\sqrt{\frac{ M  {\EE_{P} \sbr{f_L ( \lambda^*,  \mu^*; \omega) }
} \log(12B^2/(\tau \varepsilon^2 \delta))}{n}} + \frac{ M \log(12B^2/(\tau \varepsilon^2 \delta))}{n} 
} \bigg)
\nonumber \\
\overset{(b)}{\leq} &   4 \rbr{
\sqrt{\frac{ M 
(B + \underline{\lambda} \tau  )
 \log(12B^2/(\tau \varepsilon^2 \delta))}{n}} + \frac{ M \log(12B^2/(\tau \varepsilon^2 \delta))}{n} 
}, \label{ineq_bernstein_at_pop_dual}
\end{align}
where $M = (3+2L)B$, $(a)$ follows from applying  \eqref{eq_bernstein_1} in  Lemma \ref{lemma_two_side_bernstein} while using  \eqref{dual_stoch_obj_bound} in Lemma \ref{lemma_lipschitz},
and $(b)$ follows from \eqref{risk_via_emp_restrict}.

For any $\varepsilon > 0$, let $\cN_1$ (resp. $\cN_2$) be an $\varepsilon$-net of $[\underline{\lambda}, \frac{2B}{\tau}]$ (resp. $[-B, B]$). 
In addition, for any $\lambda \in  [\underline{\lambda}, \frac{2B}{\tau}]$ and $\mu \in [-B, B]$, let 
$\lambda_\varepsilon$ and $\mu_\varepsilon$ be their corresponding closest points in $\cN_1$ and $\cN_2$ respectively. 
From \eqref{eq_bernstein_2} in Lemma \ref{lemma_two_side_bernstein} and \eqref{dual_stoch_obj_bound} in Lemma \ref{lemma_lipschitz}, we have that for any $\delta \in (0,1)$, with probability $1-\delta$,
\begin{align}\label{ineq_bernstein_eps_net}
{
\EE_{P} \sbr{f_L (\lambda_\varepsilon, \mu_\varepsilon; \omega)}  - \EE_{P_n} \sbr{f_L (\lambda_\varepsilon, \mu_\varepsilon; \omega)}
}
\leq 
 4 \rbr{
\sqrt{\frac{ M  {\EE_{P_n} \sbr{f_L (\lambda_\varepsilon, \mu_\varepsilon; \omega) }
} \log(12B^2/(\tau \varepsilon^2 \delta))}{n}} + \frac{ M \log(12B^2/(\tau \varepsilon^2 \delta))}{n} 
}
\end{align}
for any $ \lambda \in [\underline{\lambda}, \frac{2B}{\tau}], \mu \in [-B, B]$.
Consequently, we obtain
\begin{align}
&  \EE_{P} \sbr{f_L (\lambda^*, \mu^*; \omega)}
- 
\EE_{P_n} \big[{f_L (\hat{\lambda}^*, \hat{\mu}^*; \omega)} \big] \nonumber \\ 
 \leq &   \EE_{P} \big[{f_L (\hat \lambda^*, \hat \mu^*; \omega)} \big] 
-  \EE_{P_n} \big[{f_L (\hat{\lambda}^*, \hat{\mu}^*; \omega)} \big]  \nonumber  \\
 \overset{(c)}{\leq} &   \EE_{P} \big[{f_L (\hat{\lambda}^*_\varepsilon, \hat{\mu}^*_\varepsilon; \omega)}\big] 
-  \EE_{P_n} \big[{f_L (\hat{\lambda}^*_\varepsilon, \hat{\mu}^*_\varepsilon; \omega)} \big] + 2 \overline{L} \varepsilon \nonumber \\
 \overset{(d)}{\leq}  &  
 4 \bigg({
\sqrt{\frac{ M  {\EE_{P_n} \sbr{f_L (\hat \lambda^*_\varepsilon, \hat \mu^*_\varepsilon; \omega) }
} \log(12B^2/(\tau \varepsilon^2 \delta))}{n}} + \frac{ M \log(12B^2/(\tau \varepsilon^2 \delta))}{n} 
} \bigg) + 2 \overline{L} \varepsilon \nonumber  \\
 \overset{(e)}{\leq}  &   
  4 \rbr{
\sqrt{\frac{ M 
(B + \underline{\lambda} \tau + \overline{L} \varepsilon)
 \log(12B^2/(\tau \varepsilon^2 \delta))}{n}} + \frac{ M \log(12B^2/(\tau \varepsilon^2 \delta))}{n} 
}  + 2 \overline{L} \varepsilon, 
\label{pop_emp_min_loss_diff_1}
\end{align}
where 
$(c)$ follows from Lemma \ref{lemma_lipschitz} and the definition of $\overline{L} =  {L+1 +  \frac{2LB}{\underline{\lambda}} + \tau}$, $(d)$ follows from \eqref{ineq_bernstein_eps_net}, and $(e)$ follows from \eqref{risk_via_emp_restrict_truncate} and Lemma \ref{lemma_lipschitz}.

Now combining \eqref{risk_via_emp_restrict}, \eqref{risk_via_emp_restrict_truncate}, \eqref{ineq_bernstein_at_pop_dual}, and \eqref{pop_emp_min_loss_diff_1}, we obtain 
\begin{align*}
\cR_{n,L}(X) & \leq \cR_L(X) + \underline{\lambda} \tau + 4 \rbr{
\sqrt{\frac{ M 
(B + \underline{\lambda} \tau  )
 \log(12B^2/(\tau \varepsilon^2 \delta))}{n}} + \frac{ M \log(12B^2/(\tau \varepsilon^2 \delta))}{n} 
}, \\
\cR_L(X) & \leq \cR_{n,L}(X) + \underline{\lambda} \tau 
+  4 \rbr{
\sqrt{\frac{ M 
(B + \underline{\lambda} \tau + \overline{L} \varepsilon)
 \log(12B^2/(\tau \varepsilon^2 \delta))}{n}} + \frac{ M \log(12B^2/(\tau \varepsilon^2 \delta))}{n} 
}  + 2 \overline{L} \varepsilon,
\end{align*}
which together with \eqref{ineq_truncate_tail_via_growth} in Lemma \ref{lemma_risk_via_truncation}, implies
\begin{align*}
\abs{ \cR_{n} (X) - \cR(X)} 
\leq 
 \underline{\lambda} \tau 
+  4 \rbr{
\sqrt{\frac{ M 
(B + \underline{\lambda} \tau + \overline{L} \varepsilon)
 \log(12B^2/(\tau \varepsilon^2 \delta))}{n}} + \frac{ M \log(12B^2/(\tau \varepsilon^2 \delta))}{n} 
}  + 2 \overline{L} \varepsilon
+ \frac{B\tau}{g_\phi(L)}.
\end{align*}
The desired claim then follows by substituting the choice of 
$\underline{\lambda} = \underline{\lambda}(\epsilon)$,
$L = L(\epsilon)$, and $\varepsilon = \frac{\epsilon}{8 \overline{L}(\epsilon)}$ into the above relation, 
while noting that $\epsilon \leq B$.
\end{proof}

\begin{remark}
In view of Theorem \ref{thrm_sample_bernstein}, the number of samples to ensure $\abs{\cR_n(X) - \cR(X)} \leq \epsilon$ is at the order of 
$\cO( B M(\epsilon)  /\epsilon^2) = \cO( L(\epsilon) B^2  /\epsilon^2)$.
This improves the sample complexity of $\cO(  M^2(\epsilon)  /\epsilon^2)$ obtained in Theorem \ref{thrm_sample_p_independent_truncation}  by an order of $\Theta( L(\epsilon) ) = \Theta (g_\phi^{-1}(\frac{B\tau}{\epsilon}))$. 
As we will see in Section \ref{sec_application}, such an improvement already yields an optimal sample complexity for estimating conditional value-at-risk $\mathrm{CVaR}_\alpha(X)$,
where the corresponding divergence function is given by $\phi(x) = \mathbbm{1}_{[0, 1/\alpha]}(x)$  \cite{shapiro2017distributionally}.

On the other hand, in view of the hard example we have constructed in Lemma \ref{lemma_risk_lb_suplinear}, it can be seen that the obtained sample complexity remains suboptimal. 
Indeed, let $\zeta^*$ be the worst-case density associated with the example in Lemma \ref{lemma_risk_lb_suplinear}, our goal is to estimate 
$
\cR(X) = \EE_P \sbr{ X \zeta^*}.
$
It is clear that $X \zeta^*$ takes the value of $B g_\phi^{-1}\rbr{\frac{\tau B}{2\epsilon}}$ with probability $\frac{\epsilon / B}{g_\phi^{-1}( \tau B / (2\epsilon))}$, and zero otherwise. 
Hence we have 
$\EE_P \sbr{(X \zeta^*)^2} \leq  \epsilon B g_\phi^{-1}({\tau B}\rbr{2\epsilon})  = \cO(\epsilon M(\epsilon))$. 
Hence the Bernstein inequality would suggest that the sample complexity for estimating 
$\cR(X)$
would scale at the order of $ \cO(\epsilon M(\epsilon))$, instead of $\cO( B M(\epsilon))$ prescribed by Theorem~\ref{thrm_sample_bernstein}.
\end{remark}

The above discussion then suggests that our control of $\EE_P \big[{f_{L}^2 (\lambda^*,\mu^*; \omega)}\big]$ can be loose. 
At the same time, in view of our above discussion, for the hard example constructed in Lemma \ref{lemma_risk_lb_suplinear}, the worst-case density $\zeta^*$ would take a large value with a small probability that depends on the growth function $g_\phi$, and therefore one can readily control 
$\EE_P \sbr{(\zeta^*)^2}$.
We proceed to exploit this observation and provide a refined control on $\EE_P \big[{f_{L}^2 (\lambda^*,\mu^*; \omega)}\big]$.
To this end, we first make the following structural observation that $f_L(\lambda^*, \mu^*; \omega)$ can be described by the worst-case density $\zeta^*$ associated with a perturbed version of the truncated risk \eqref{phi_d_truncate}.

\begin{lemma}\label{lemma_rep_f_via_zeta}
For any $L \geq 1$, there exists density $\zeta^*$ with 
\begin{align}\label{rep_f_via_zeta}
f_L(\lambda^*, \mu^*; \omega) 
= \lambda^* \tau + \mu^* + (X(\omega) - \mu^*) \zeta^*(\omega) - (\lambda^* \phi_L)(\zeta^*(\omega)), ~~~ \text{$P$-almost surely},
\end{align}
and $\zeta^*$ is feasible to \eqref{phi_d_truncate}.
\end{lemma}

\begin{remark}
It is quite natural to see that \eqref{rep_f_via_zeta} holds in the special case of $\underline{\lambda} = 0$ and $L > 1$. 
From \eqref{eq_phi_dro_f_form_restrict} in Lemma \ref{lemma_bounded_dual}, in this case $(\lambda^*, \mu^*)$ defined in \eqref{eq_def_optimal_dual}  is the optimal dual solution of \eqref{phi_d_truncate}.
Combining this with the fact that strong duality holds for \eqref{phi_d_truncate}, the optimal density $\zeta^*$  of \eqref{phi_d_truncate} must minimize $\cL(\zeta; \lambda^*, \mu^*)$,  from which \eqref{rep_f_via_zeta} holds with $\zeta^*(\omega) \in \Argmax_{z} (X(\omega) - \mu^*) z - \lambda^* \phi_L(z)$ $P$-almost surely. 
Lemma \ref{lemma_rep_f_via_zeta}  can hence be viewed as a generalization of this simple observation to the case of $\underline{\lambda} > 0$. 
In particular, we show that  $(\lambda^*, \mu^*)$ is the dual solution of a perturbed version of \eqref{phi_d_truncate} when $\underline{\lambda} > 0$.
\end{remark} 

\begin{proof}
We first proceed to show that $(\lambda^* - \underline{\lambda}, \mu^*)$ is the optimal dual solution for 
\begin{align}\label{eq_def_shifted_dro}
\sup_{\zeta \in \mathfrak{D}} \EE_P \sbr{X \zeta} + \underline{\lambda} \rbr{ \tau - \EE_P \sbr{\phi_L(\zeta)}}, ~~ \mathrm{s.t.} ~~ \int_\Omega \phi_L(\zeta(\omega)) dP(\omega) \leq \tau. 
\end{align}
Indeed, the Lagrangian of the above problem is given by 
\begin{align*}
\cL(\zeta; \lambda', \mu') =  (\underline{\lambda} + \lambda') \tau + \mu' + \EE_P [ (X - \mu')\zeta - (\underline{\lambda} + \lambda')\phi_L(\zeta) ] .
\end{align*}  
Hence the dual is  given by 
\begin{align}
\inf_{\lambda' \geq 0, \mu \in \RR} \sup_{\zeta \in  \cL_1(P)} \cL(\zeta; \lambda', \mu')  & \overset{(a)}{=} 
\inf_{\lambda' \geq 0, \mu \in \RR} (\underline{\lambda} + \lambda') \tau + \mu' + \EE_P [ \sup_z (X-\mu') z - (\underline{\lambda} + \lambda')\phi_L(z) ]  \nonumber \\
%& = \inf_{\lambda' \geq 0, \mu' \in \RR} (\underline{\lambda} + \lambda') \tau + \mu' + \EE_P \sbr{ ((\underline{\lambda} + \lambda')\phi_L)^* (X - \mu') } \nonumber \\
& = \inf_{\lambda \geq \underline{\lambda}, \mu \in \RR} \lambda \tau + \mu + \EE_P \sbr{ (\lambda \phi_L)^* (X - \mu) }, \label{def_shift_dual} 
\end{align}
where $(a)$ follows from the fact that  $\zeta \in \cL_1(P)$ and the latter is decomposable, and hence one can interchange taking the expectation and  supremum \cite{rockafellar1998variational}.
Following the same lines as in Lemma \ref{lemma_bounded_dual}, 
for any $\underline{\lambda} \leq 2B/\tau$, 
it can be readily shown that the above problem is further equivalent to 
\begin{align*}
 \inf_{\lambda \geq \underline{\lambda}, \mu \in \RR} \lambda \tau + \mu + \EE_P \sbr{ (\lambda \phi_L)^* (X - \mu) }
 = \inf_{\lambda \in [\underline{\lambda}, \frac{2B}{\tau}], \mu \in [-B, B]} \lambda \tau + \mu + \EE_P \sbr{ (\lambda\phi_L)^* (X - \mu) }.
\end{align*}
Combining the above relation and \eqref{def_shift_dual},
together with the definition of $(\lambda^*, \mu^*)$ in \eqref{eq_def_optimal_dual},
 we conclude that $(\lambda^*- \underline{\lambda}, \mu^*)$ is the dual solution for \eqref{eq_def_shifted_dro}. 
 
 Now since the ambiguity set in \eqref{eq_def_shifted_dro} is weakly compact in $\cL_1(P)$, 
 and 
 $\EE_P \sbr{X \zeta} + \underline{\lambda} \rbr{ \tau - \EE_P \sbr{\phi_L(\zeta)}}$ is lower-semicontinuous, 
  the supremum in \eqref{eq_def_shifted_dro} is attainable at some $\zeta^*$ feasible to \eqref{eq_def_shifted_dro}. 
%If $L > 1$, then from Slater condition, strong duality holds for \eqref{def_shift_dual}, 
If strong duality holds for \eqref{def_shift_dual}, 
it follows that  
 $(\zeta^*, \lambda^* - \underline{\lambda}, \mu^*)$ forms a saddle point of \eqref{def_shift_dual}, and hence 
 $\zeta^* \in \Argmax_{\zeta} \cL(\zeta; \lambda^* - \underline{\lambda}, \mu^*)$.
 Consequently, we obtain
 \begin{align*}
\lambda^* \tau + \mu^* + \EE_P [ (X-\mu^*) \zeta^* -\lambda^* \phi_L(\zeta^*) ]
= \lambda^* \tau + \mu^* + \EE_P \sbr{ (\lambda^* \phi_L)^* (X - \mu^*) }.
 \end{align*}
Combining the above observation with the definition of $(\lambda^* \phi_L)^*$, we must have 
\begin{align*}
 (\lambda^* \phi_L)^* (X(\omega) - \mu^*)  = (X(\omega) -\mu^*) \zeta^*(\omega) -\lambda^* \phi_L(\zeta^* (\omega)) 
\end{align*}
$P$-almost surely.
It remains to show that \eqref{def_shift_dual} satisfies strong duality for $L \geq 1$. For $L > 1$ this clearly follows from the Slater condition. 
On the other hand, if $L = 1$, then $\mathrm{OPT}\eqref{def_shift_dual} = \EE_P \sbr{X}  + \underline{\lambda} \tau$.
We proceed to show that the corresponding dual with $((\lambda')^*, (\mu')^*) = (0,0)$ attains the same value. 
Indeed, we have 
\begin{align*}
\sup_{\zeta \in \cL_1(P)} \cL(\zeta; 0,0) \overset{(b)}{=} \underline{\lambda} \tau + \EE_P \sbr{(\underline{\lambda} \phi_1)^*(X)} 
=  \underline{\lambda} \tau + \EE_P \big[{\sup_x Xx - \underline{\lambda} \phi_1(x)} \big] \overset{(c)}{=} \lambda \tau + \EE_P \sbr{X},
\end{align*}
where $(b)$ follows from \eqref{def_shift_dual}, and $(c)$ follows from the fact that $X \geq 0$ and $1$ is a minimizer of $\phi_1$. 
\end{proof}

It could worth noting that Lemma \ref{lemma_rep_f_via_zeta} holds for any distribution $P$ over $(\Omega, \cF)$, in particular, the empirical distribution $P_n$. 
In this case, \eqref{rep_f_via_zeta} holds with $(\lambda^*, \mu^*)$ replaced by $(\hat{\lambda}^*, \hat{\mu}^*)$. 
With Lemma \ref{lemma_rep_f_via_zeta} in place, we can now provide a refined characterization on $\EE_P \big[{f_{L}^2 (\lambda^*,\mu^*; \omega)}\big]$ and 
$\EE_{P_n} \big[{f_{L}^2 (\hat{\lambda}^*,\hat{\mu}^*; \omega)}\big]$.

\begin{lemma}\label{lemma_refined_variance}
For any $L \geq L(B)$, define\footnote{
We define convention $ \sup_{t \in (L(B), L] } {t}/{g_\phi (t)} = 0$ if $L = L(B)$.
Note that in this case \eqref{ineq_second_moment_worst_case_f} holds trivially. 
} 
\begin{align}\label{def_var_proxy}
V_L = B^2 \bigg({9 + 4 L(B) + 4 \tau \sup_{t \in (L(B), L] } \frac{t}{g_\phi (t)}} \bigg).
\end{align}
Then we have 
\begin{align}\label{ineq_refined_var}
 \max \cbr{\EE_P \sbr{ f_{L}^2 (\lambda^*,\mu^*; \omega) }, \EE_{P_n} \big[{ f_{L}^2 (\hat{\lambda}^*,\hat{\mu}^*; \omega)  \big]} 
}  \leq V_L.
\end{align}
\end{lemma}

\begin{proof}
From \eqref{rep_f_via_zeta} in Lemma \ref{rep_f_via_zeta}, the definition of $(\lambda^*, \mu^*)$ in \eqref{eq_def_optimal_dual}, and \eqref{f_nonnegative}, we have  
\begin{align}\label{f_ub_zeta_dependent}
\abs{f_L(\lambda^*, \mu^*; \omega) }
\leq 3B + 2B  \zeta^*(\omega)
\end{align}
$P$-almost surely.
\begin{align*}
\EE \sbr{ (\zeta^*(\omega) )^2} 
&= \int_{\zeta^*(\omega) \leq L(B)} (\zeta^*(\omega) )^2   dP(\omega) + \int_{\zeta^*(\omega) > L(B)} (\zeta^*(\omega) )^2   dP(\omega)  \\
& \leq L(B) + \int_{ L(B) < \zeta^*(\omega) \leq L } \frac{\zeta^*(\omega)}{\phi(\zeta^*(\omega))} \zeta^*(\omega) \phi(\zeta^*(\omega)) dP(\omega)  \\
& \overset{(a)}{\leq} L(B) +  \int_{ L(B) < \zeta^*(\omega) \leq L } \frac{\zeta^*(\omega)}{g_\phi(\zeta^*(\omega))}  \phi_L(\zeta^*(\omega)) dP(\omega) \\
&  \overset{(b)}{\leq}  L(B) +   \tau    \sup_{t \in (L(B), L] } \frac{t}{g_\phi (t)},
\end{align*}
where $(a)$ follows from \eqref{g_phi_simplified} and the fact that $g_\phi(x) > 0$ if $x > L(B)$ given the definition of $L(B)$ and $g_\phi^{-1}$. 
Consequently, by combining the above observation with \eqref{f_ub_zeta_dependent}, we obtain 
\begin{align}\label{ineq_second_moment_worst_case_f}
\EE_P \sbr{ f_{L}^2 (\lambda^*,\mu^*; \omega) } 
% \leq 
%\EE_P \sbr{f_{L}^2(\lambda^*,\mu^*; \omega)}  
&  \leq \EE_P \sbr{ 9B^2 + 4B^2 (\zeta^*)^2 (\omega)}  \leq
B^2 \bigg({
9 + 4 L(B) + 4 \tau   \sup_{t \in (L(B), L] } \frac{t}{g_\phi (t)}
} \bigg).
\end{align}
Hence $\EE_P \sbr{ 9B^2 + 4B^2 (\zeta^*)^2 (\omega)} \leq V_L$. 
The proof for $\EE_{P_n} \big[{ f_{L}^2 (\hat{\lambda}^*,\hat{\mu}^*; \omega)}\big]  \leq V_L$ follows from the same lines as above, after noting that Lemma \ref{lemma_rep_f_via_zeta} also holds by replacing $P$ (resp. $(\lambda^*, \mu^*)$) with 
$P_n$ (resp. $(\hat{\lambda}^*, \hat{\mu}^*)$).  
\end{proof}

By  combining Lemma \ref{lemma_refined_variance} with Lemma \ref{lemma_two_side_bernstein},  we are now ready to establish an improved sample complexity for $\cR_n(X)$  compared to Theorem \ref{thrm_sample_bernstein}.

\begin{theorem}\label{thrm_sample_complexity_improved_var}
Suppose $\lim_{x \to \infty} \phi(x) / x = \infty$. 
For any $0 < \epsilon \leq B $, let 
$\underline{\lambda}(\epsilon), L(\epsilon)$, 
 $\overline{L}(\epsilon)$ and $M(\epsilon)$ be defined as in \eqref{eq_lip_m_constants}.
Then 
for any $\delta \in (0,1)$  
we have $\abs{\cR(X) - \cR_n(X)} \leq \epsilon$ with probability at least $1-\delta$
 if 
\begin{align*}
n \geq 
\frac{4096 (V_{L(\epsilon)} + B^2) \log (768 B^2 \overline{L}^2(\epsilon)/ (\epsilon^2 \tau \delta))}{\epsilon^2} 
+ 
\frac{64 M(\epsilon)  \log (768 B^2 \overline{L}^2(\epsilon)/ (\epsilon^2 \tau \delta))}{\epsilon} ,
\end{align*}
where $V_L$ is defined as in \eqref{def_var_proxy}.
\end{theorem}

\begin{proof}
Combining Lemma \ref{lemma_refined_variance} with \eqref{eq_bernstein_1} in Lemma \ref{lemma_two_side_bernstein}, while noting that \eqref{dual_stoch_obj_bound} still holds,
we obtain that for any $\delta \in (0,1)$, with probability at least $1-\delta$,
\begin{align}
 \EE_{P_n} \big[{f_L(\hat{\lambda}^*, \hat{\mu}^*; \omega) }\big] - \EE_{P} \sbr{f_L(\lambda^*, \mu^*; \omega) }  &
\leq  
 \EE_{P_n} \big[{f_L(\hat{\lambda}^*, \hat{\mu}^*; \omega) }\big] - \EE_{P} \sbr{f_L(\lambda^*, \mu^*; \omega) } 
\nonumber  \\
& \leq   4 
\bigg({
 \sqrt{\frac{ V_L \log(3/\delta)}{n}} + \frac{B (3 + 2L) \log(3/\delta)}{n} 
} \bigg). \label{ineq_row_diff_pop_improved}
\end{align}

For any $\varepsilon > 0$, let $\cN_1$ (resp. $\cN_2$) be an $\varepsilon$-net of $[\underline{\lambda}, \frac{2B}{\tau}]$ (resp. $[-B, B]$). 
In addition, for $\lambda \in  [\underline{\lambda}, \frac{2B}{\tau}]$ and $\mu \in [-B, B]$, we denote  
$\lambda_\varepsilon$ and $\mu_\varepsilon$ as their corresponding closest points in $\cN_1$ and $\cN_2$ respectively. 
Combining  \eqref{dual_stoch_obj_bound} and \eqref{eq_bernstein_2} in Lemma \ref{lemma_two_side_bernstein}, we have 
\begin{align*}
\EE_{P} \sbr{f_L({\lambda}_\varepsilon, {\mu}_\varepsilon; \omega) } -  \EE_{P_n} \sbr{f_L({\lambda}_\varepsilon, {\mu}_\varepsilon; \omega) }  
\leq 
4 \bigg({
\sqrt{\frac{  { \EE_{P_n} \sbr{f_L^2({\lambda}_\varepsilon, {\mu}_\varepsilon; \omega)}} \log(12 B^2/(\tau \varepsilon^2 \delta))}{n}} + \frac{(3 + 2L)B \log(12 B^2/(\tau \varepsilon^2 \delta)}{n} 
} \bigg)
\end{align*}
for any $\lambda \in  [\underline{\lambda}, \frac{2B}{\tau}]$ and $\mu \in [-B, B]$.
This in turn implies 
\begin{align}
& \EE_{P} \big[{f_L({\lambda}^*, {\mu}^*; \omega) }\big] -  \EE_{P_n} \big[{f_L(\hat{\lambda}^*, \hat{\mu}^*; \omega) } \big]\nonumber  \\
\leq & \EE_{P} \big[{f_L(\hat{\lambda}^*, \hat{\mu}^*; \omega) }\big] -  \EE_{P_n} \big[{f_L(\hat{\lambda}^*, \hat{\mu}^*; \omega) } \big]\nonumber  \\
 \overset{(a)}{\leq}   & 
\EE_{P} \big[ {f_L(\hat{\lambda}^*_\varepsilon, \hat{\mu}^*_\varepsilon; \omega) } \big] -  \EE_{P_n} \big[{f_L(\hat{\lambda}^*_\varepsilon, \hat{\mu}^*_\varepsilon; \omega) } \big] + 2 \overline{L} \varepsilon   \nonumber \\
\leq & 
4 \bigg({
\sqrt{\frac{  { \EE_{P_n} \sbr{f_L^2(\hat{\lambda}^*_\varepsilon, \hat{\mu}^*_\varepsilon; \omega)}} \log(12 B^2/(\tau \varepsilon^2 \delta))}{n}} + \frac{(3 + 2L)B \log(12 B^2/(\tau \varepsilon^2 \delta)}{n} 
} \bigg)
+ 2 \overline{L} \varepsilon   \nonumber  \\
\overset{(b)}{\leq} & 
8 \bigg({
\sqrt{\frac{  { \EE_{P_n} \sbr{f_L^2(\hat{\lambda}^*, \hat{\mu}^*; \omega) + \overline{L}^2 \varepsilon^2 }} \log(12 B^2/(\tau \varepsilon^2 \delta))}{n}} + \frac{(3 + 2L)B \log(12 B^2/(\tau \varepsilon^2 \delta)}{n} 
} \bigg)
+ 2 \overline{L} \varepsilon \nonumber \\
\overset{(c)}{\leq} & 
8 \bigg({
\sqrt{\frac{  {  \rbr{V_L + \overline{L}^2 \varepsilon^2 }} \log(12 B^2/(\tau \varepsilon^2 \delta))}{n}} + \frac{(3 + 2L)B \log(12 B^2/(\tau \varepsilon^2 \delta)}{n} 
} \bigg)
+ 2 \overline{L} \varepsilon, \label{ineq_row_diff_emp_improved}
\end{align}
where $(a)$ and $(b)$ follow from  Lemma \ref{lemma_lipschitz} and the definition of $\overline{L} =  {L+1 +  \frac{2LB}{\underline{\lambda}} + \tau}$,
and $(c)$ follows from Lemma \ref{lemma_refined_variance}.

Combining \eqref{risk_via_emp_restrict} and \eqref{risk_via_emp_restrict_truncate} with \eqref{ineq_row_diff_pop_improved} and \eqref{ineq_row_diff_emp_improved} above, we have that for any $\varepsilon > 0$ and $L \geq L(B)$, 
with probability at $1-\delta$,
\begin{align*}
\cR_{n, L}(X) & \leq \cR(X) + \underline{\lambda} \tau + 4 
\bigg({
B \sqrt{\frac{ V_L \log(3/\delta)}{n}} + \frac{B (3 + 2L) \log(3/\delta)}{n} 
} \bigg), \\ 
\cR_L(X) & \leq \cR_{n, L}(X) +  \underline{\lambda}\tau + 
8 \bigg({
\sqrt{\frac{  { \rbr{V_L + \overline{L}^2 \varepsilon^2 }} \log(12 B^2/(\tau \varepsilon^2 \delta))}{n}} + \frac{(3 + 2L)B \log(12 B^2/(\tau \varepsilon^2 \delta)}{n} 
} \bigg)
+ 2 \overline{L} \varepsilon, 
\end{align*}
where $V_L$ is defined as in \eqref{def_var_proxy}.
The above observation, together with Lemma \ref{lemma_risk_via_truncation}, then implies that 
\begin{align*}
\abs{\cR_n(X) - \cR(X)} 
\leq 
 \underline{\lambda}\tau + 
8 \bigg({
\sqrt{\frac{  { \rbr{V_L + \overline{L}^2 \varepsilon^2 }} \log(12 B^2/(\tau \varepsilon^2 \delta))}{n}} + \frac{(3 + 2L)B \log(12 B^2/(\tau \varepsilon^2 \delta)}{n} 
} \bigg)
+ 2 \overline{L} \varepsilon
+ \frac{B\tau}{g_\phi(L)}.
\end{align*}
The desired claim then follows by substituting the choice of 
$\underline{\lambda} = \underline{\lambda}(\epsilon)$,
$L = L(\epsilon)$, and $\varepsilon = \frac{\epsilon}{8 \overline{L}(\epsilon)}$ into the above relation, 
while noting that $\epsilon \leq B$.
\end{proof}

To compare the obtained sample complexity of Theorem \ref{thrm_sample_complexity_improved_var} and that of Theorem \ref{thrm_sample_bernstein}, we make the simple observation.  

\begin{corollary}\label{corr_sample_improved_var}
Suppose $\lim_{x \to \infty} \phi(x) / x = \infty$. 
For any $0 < \epsilon \leq B $, let 
$\underline{\lambda}(\epsilon), L(\epsilon)$, 
 $\overline{L}(\epsilon)$ and $M(\epsilon)$ be defined as in \eqref{eq_lip_m_constants}.
Then 
for any $\delta \in (0,1)$  
we have $\abs{\cR(X) - \cR_n(X)} \leq \epsilon$ with probability at least $1-\delta$
 if 
\begin{align*}
n \geq 
\frac{32768  \sup_{\epsilon' \in [\epsilon, B]}  M(\epsilon') \epsilon'   \log (768 B^2 \overline{L}^2(\epsilon)/ (\epsilon^2 \tau \delta))}{\epsilon^2} 
+ 
\frac{64 M(\epsilon)  \log (768 B^2 \overline{L}^2(\epsilon)/ (\epsilon^2 \tau \delta))}{\epsilon} ,
\end{align*}
\end{corollary}

\begin{proof}
Let us first consider the case where $L(\epsilon) > L(B)$, and denote $\epsilon_t = \frac{4 B\tau}{g_\phi(t)}$ for $t \in (L(B), L(\epsilon)]$. 
For any $t > L(B)$, we have $g_\phi(t) > 4\tau$ from the definition of $L(B) $ and Proposition \ref{prop_inverse_g}.      
Combining this with Remark \ref{remark_monotone_growth}, we have $g(t) > 0$ for $t \in (L(B), L(\epsilon)]$ and is strictly increasing over this domain.
This in turn implies 
$t = g_\phi^{-1}(\frac{4 B\tau}{\epsilon_t})$ for $t \in (L(B), L(\epsilon)]$, and therefore 
\begin{align}\label{ineq_change_var_raw}
 \sup_{t \in (L(B), L(\epsilon)] } \frac{t}{g_\phi (t)}
 =  \sup_{t \in (L(B), L(\epsilon)] }  g_\phi^{-1} (\frac{4 B\tau}{\epsilon_t}) \frac{\epsilon_t}{4B\tau}.
\end{align}
In addition,  for any $t > L(B)$, 
we have $g_\phi(t) > 4 \tau$  and hence  $\epsilon_t = \frac{4B\tau}{g_\phi(t)} < B$.
On the other hand, when $t = L(\epsilon)$, we have $g_\phi(L(\epsilon)) \leq \frac{4B\tau}{\epsilon}$ following again from Proposition \ref{prop_inverse_g}, and subsequently 
$\epsilon_t = \frac{4B\tau}{g_\phi (L(\epsilon))} \geq \epsilon$. 
Combining the above two observations with \eqref{ineq_change_var_raw}, we obtain 
\begin{align*}
 \sup_{t \in (L(B), L(\epsilon)] } \frac{t}{g_\phi (t)}
\leq \sup_{\epsilon' \in [\epsilon, B]}  g_\phi^{-1} (\frac{4 B\tau}{\epsilon'}) \frac{\epsilon'}{4B\tau}
= \sup_{\epsilon' \in [\epsilon, B]}  L(\epsilon') \frac{\epsilon'}{4B\tau}.
\end{align*}
Hence from the definition of $V_L$ in \eqref{def_var_proxy}, we obtain 
\begin{align*}
V_{L(\epsilon)} + B^2  = B^2 \bigg(10 + 4 L(B) +   \sup_{\epsilon' \in [\epsilon, B]}  L(\epsilon') \frac{\epsilon'}{B} \bigg)
%\leq 20 \log\rbr{{L(\epsilon)}}  \sup_{\epsilon' \in [\epsilon, B]}  B L(\epsilon')\epsilon'
\leq 8   \sup_{\epsilon' \in [\epsilon, B]}  M(\epsilon') \epsilon',
\end{align*}
where the last inequality follows from the definition of $M(\epsilon)$. 
Now if $L(\epsilon) = L(B)$, then $V_{L(\epsilon)} = B^2 (10 + 4 L(B))$  in view of \eqref{def_var_proxy}, and the above relation still holds.
The proof is then concluded by invoking Theorem~\ref{thrm_sample_complexity_improved_var}.
\end{proof}

For any precision target $\epsilon > 0$, Corollary \ref{corr_sample_improved_var} suggests that it takes no more than 
$
 \tilde{\cO} \rbr{ 
\frac{  \sup_{\epsilon' \in [\epsilon, B]}  M(\epsilon') \epsilon'  }{\epsilon^2}  + \frac{M(\epsilon)}{\epsilon}
}
$
number of samples to estimate $\cR_n(X)$.  
Given the fact $M(\epsilon)$ is non-increasing in $\epsilon$, 
this clearly improves upon the sample complexity of 
$
\cO \rbr{ 
\frac{ B M(\epsilon)  }{\epsilon^2} + \frac{M(\epsilon)}{\epsilon}
} 
$
prescribed by Theorem \ref{thrm_sample_bernstein}.
It could be worth mentioning that the sample complexities we obtained in Theorem \ref{thrm_sample_p_independent_truncation}, \ref{thrm_sample_bernstein}, \ref{thrm_sample_complexity_improved_var} and Corollary \ref{corr_sample_improved_var} are all invariant to scaling of the random variable $X$.
In the ensuing Section \ref{sec_application}, we will proceed to instantiate the obtained sample complexities in Section \ref{sec_sublinear} and \ref{sec_suplinear},   and establish their optimality for commonly used $\phi$-divergences.

%!TEX root = ./sample_dro.tex

\newpage 

\section{Application to Common $\phi$-divergences}\label{sec_application}

We are now ready to instantiate the sample complexity bounds obtained in Section \ref{sec_sublinear} and \ref{sec_suplinear} for some commonly used $\phi$-divergences.
In particular, interested readers may refer to  \cite[Table 1]{bayraksan2015data} for a comprehensive collection of $\phi$-divergences.
We begin by showing that for a class of $\phi$-divergences, obtaining an $\epsilon$-accurate estimation of $\cR(X)$ can require an arbitrarily large number of samples. 

\begin{corollary}[Essential supremum]
For $\cR(X) = \esssup X$, obtaining an $\epsilon$-accurate estimation of $\cR(X)$ can require an arbitrarily large number of samples.
\end{corollary}

\begin{proof}
Since $\cR(X) = \esssup X$ corresponds to choosing $\phi(x) = \mathrm{1}_{[0, \infty)}(x)$ in \eqref{dro_phi},
the desired claim follows from invoking Theorem \ref{info_lb_sublinear}.
\end{proof}

\begin{corollary}[Burg entropy]
For $\cR(X)$ corresponding to the Burg entropy, obtaining an $\epsilon$-accurate estimation of $\cR(X)$ can require an arbitrarily large number of samples.
\end{corollary}

\begin{proof}
The desired claim follows from invoking Theorem~\ref{info_lb_sublinear}, and the fact that $\phi(x) = (-\log x + x - 1)\mathbbm{1}_{(0, \infty)}(x)$ for the Burg entropy.
\end{proof}

\begin{corollary}[Neyman $\chi^2$-divergence]
For $\cR(X)$ corresponding to the Neyman $\chi^2$-divergence, obtaining an $\epsilon$-accurate estimation of $\cR(X)$ can require an arbitrarily large number of samples.
\end{corollary}

\begin{proof}
The desired claim follows from invoking Theorem~\ref{info_lb_sublinear}, and the fact that  
 $\phi(x) = {\frac{(x-1)^2}{x} \mathbbm{1}_{(0, \infty)}(x)}$ for the Neyman $\chi^2$-divergence.
\end{proof}

\begin{corollary}[Total variation]
For $\cR(X)$ corresponding to the total variation distance, obtaining an $\epsilon$-accurate estimation of $\cR(X)$ can require an arbitrarily large number of samples.
\end{corollary}

\begin{proof}
The desired claim follows from invoking Theorem~\ref{info_lb_sublinear}, and the fact that  $\phi(x) = \abs{x-1} \mathbbm{1}_{[0, \infty)}(x)$ for the total
variation distance.
\end{proof}

\begin{corollary}[Hellinger distance]
For $\cR(X)$ corresponding to the Hellinger distance, obtaining an $\epsilon$-accurate estimation of $\cR(X)$ can require an arbitrarily large number of samples.
\end{corollary}

\begin{proof}
The desired claim follows from invoking Theorem~\ref{info_lb_sublinear}, and the fact  that 
 $\phi(x) = (\sqrt{x} - 1)^2 \mathbbm{1}_{[0, \infty)}(x)$ for the Hellinger distance.
\end{proof}

\begin{remark}\label{remark_sublinear_generalized_phi}
It is important to note here that the statistical intractability for $\phi$-divergences with sublinear growth, in particular,  total variation distance and the Hellinger distance, crucially hinge upon the absolute continuity requirement of $Q \ll P$ for every $Q \in \mathfrak{M}_\tau (P)$ in \eqref{dro_phi}.
To illustrate this, suppose $\Omega$ is finite, then $\mathfrak{M}_\tau (P) = \cbr{Q \ll P: \frac{1}{2} \sum_{\omega \in \Omega} \abs{Q(\omega) - P(\omega)}  \leq \tau}$ for the total variation distance,
and 
$\mathfrak{M}_\tau (P) = \big\{Q \ll P: \frac{1}{2} \sum_{\omega \in \Omega} (\sqrt{Q(\omega) } - \sqrt{P(\omega)})^2  \leq \tau \big\}$.
In many applications, on the other hand, both distances are defined without the requirement of $Q \ll P$. 
Indeed, for divergence functions $\phi$ with sublinear growth, one can consider the generalized $\phi$-divergence first introduced in \cite{csiszar1963informationstheoretische,ali1966general}, 
which recovers the familiar form of total variance and Hellinger distance without absolute continuity requirement. 
It is worth noting that the dual of \eqref{dro_phi} with the generalized $\phi$-divergence takes the same form as \eqref{eq_phi_dro_f_form}, albeit with an additional constraint \cite{kuhn2025distributionally}. 
Nevertheless,  by adopting the same technique developed in Section \ref{sec_suplinear}, one can  develop sample complexity results for divergence functions with even sublinear growth.\footnote{
To the best of our knowledge,  sample complexities have been developed only  for total variation distance \cite{xu2023improved,yang2022toward,shi2023curious} among all $\phi$-divergences with sublinear growth.
} 
In particular, in this case the sample complexity would not depend on the underlying nominal measure $P$.
This observation hence highlights the role of absolute continuity in impacting the statistical tractability of estimating the risk functional $\cR(X)$.
\end{remark}

We next proceed to establish sample complexity upper bounds for divergences with superlinear growth,  and demonstrate their optimality by comparing with the lower bound in \eqref{eq_sample_lb_suplinear_max}.

\begin{corollary}[Conditional value-at-risk]\label{corr_cvar}
Let $\cR(X) = \mathrm{CVaR}_\alpha (X)$. 
For any $\delta \in (0,1)$, the number of samples $n$ to ensure $\abs{\cR(X) - \cR_n(X)} \leq \epsilon$ with probability $1-\delta$ can be bounded by 
\begin{align}\label{ineq_cvar_sample}
\tilde{\cO} \rbr{
\frac{ B^2 }{ \alpha \epsilon^2} 
+ 
\frac{B}{\alpha \epsilon} 
}.
\end{align}
\end{corollary}

\begin{proof}
The conditional value-at-risk $\cR(X) = \mathrm{CVaR}_\alpha(X)$ corresponds to $\phi(x) = \mathbbm{1}_{[0, 1/\alpha]}(x)$. 
In view of Remark \ref{remark_monotone_growth}, we have 
$g_\phi(x) = \frac{\phi(x) }{x} = \mathbbm{1}_{[0, 1/\alpha]}(x)$ for $x \geq 1$ and subsequently $g_\phi^{-1}(y) = \frac{1}{\alpha}$ for any $y > 0$. 
The desired claim then follows from  Corollary \ref{corr_sample_improved_var} with simple verifications. 
\end{proof}

The reported upper bound in Corollary \ref{corr_cvar} can also be obtained based on the concentration of $\cR_n(X)$ for conditional value-at-risk \cite{wang2010deviation}.
In contrast to  the approach tailored for the conditional value-at-risk in \cite{wang2010deviation}, \eqref{ineq_cvar_sample} follows directly from our established sample complexity applicable to general $\phi$-divergences. 
From \eqref{eq_sample_lb_suplinear_max}, the sample complexity lower bound is given by 
$
\Omega \big( {
 \frac{B^2}{\epsilon^2} + \frac{ B }{  \alpha \epsilon  }
} \big)
$,
which differs from the obtained upper bound in Corollary \ref{corr_cvar} by a factor of ${1}/{\alpha}$. 
It is not difficult to refine our argument in Theorem \ref{thrm_sample_quadratic_lb_suplinear} for conditional value-at-risk and obtain a sample complexity lower bound of 
$
\Omega( \frac{B^2}{\alpha \epsilon^2})$.
To see this, note that for conditional value-at-risk, we have $\cR(p) = B \min \cbr{\frac{p}{\alpha}, 1}$, where $\cR(p)$ is defined in the proof of 
Theorem \ref{thrm_sample_quadratic_lb_suplinear}.
Consequently one can readily find a close pair $(p', p)$ with $\cR(p) = B p /\alpha$ and $\cR(p') = B p'/\alpha$. 
Then the lower bound of $\Omega( \frac{B^2}{\alpha \epsilon^2})$ follows from similar lines as in Theorem \ref{thrm_sample_quadratic_lb_suplinear}.
Combining this observation with Corollary \ref{corr_cvar}, the obtained sample complexity for conditional value-at-risk is indeed~optimal.
We

\begin{corollary}[Kullback-Leibler divergence]\label{corr_kl}
Let $\cR(X) = \sup \cbr{\EE_Q \sbr{X}: \mathrm{KL}(Q \Vert P) \leq \tau}$. 
For any $\delta \in (0,1)$, the number of samples $n$ to ensure $\abs{\cR(X) - \cR_n(X)} \leq \epsilon$ with probability $1-\delta$ can be bounded by 
\begin{align*}
\tilde{\cO} \rbr{
\frac{B^2 \exp(4\tau B/\epsilon)}{\epsilon^2}
}.
\end{align*}
\end{corollary}

\begin{proof}
We have $\phi(x) = \rbr{x\log x - x +1}\mathbbm{1}_{[0, \infty)}(x)$
and subsequently $g_\phi(x) = \log x - 1 + 1/x$ for $x \geq 1$. 
This in turn implies $g_\phi^{-1}(y) =\Theta(\exp(y))$ for any $y \geq 0$. 
The desired claim then follows from Corollary \ref{corr_sample_improved_var} with simple verifications. 
\end{proof}

 By instantiating  \eqref{eq_sample_lb_suplinear_max} 
 with $g_\phi^{-1}(y) =\Theta(\exp(y))$, and comparing it with 
 Corollary \ref{corr_kl}, it can be seen that both the sample complexity upper and lower bounds for the Kullback-Leibler divergence exhibit an exponential dependence of $\Theta (\exp(B\tau/\epsilon))$ on the precision target  $\epsilon$.
It appears to us that Corollary \ref{corr_kl} provides the first sample complexity bound for estimating distributionally robust functional generated by the  Kullback-Leibler divergence, without any dependence on the nominal measure. 
%The obtained complexity bound also matches the lower bound and hence is optimal.
 It is perhaps worth mentioning that the obtained results do not require the sample space to be finite (cf., \cite{xu2023improved, shi2024distributionally,wang2024sample}).

\begin{corollary}[Jeffreys divergence]\label{corr_jeffreys}
Let $\cR(X) = \sup \cbr{\EE_Q \sbr{X}: \mathrm{KL}(Q \Vert P) + \mathrm{KL}(P \Vert Q) \leq \tau}$. 
For any $\delta \in (0,1)$, the number of samples $n$ to ensure $\abs{\cR(X) - \cR_n(X)} \leq \epsilon$ with probability $1-\delta$ can be bounded by 
\begin{align*}
\tilde{\cO} \rbr{
\frac{B^2 \exp(4\tau B/\epsilon)}{\epsilon^2}
}.
\end{align*}
\end{corollary}

\begin{proof}
We have $\phi(x) = \log x (x-1) \mathbbm{1}_{(0,\infty)}(x)$
and subsequently $g_\phi(x) = (1-1/x)\log x$ for $x \geq 1$. 
Similar to Corollary \ref{corr_kl}, this in turn implies $g_\phi^{-1}(y) = \Theta(\exp(y))$ for any $y \geq 0$. 
The desired claim then follows from Corollary \ref{corr_sample_improved_var} with simple verifications. 
\end{proof}

It is interesting to note that for the Jefferys divergence, the corresponding conjugate of the divergence function $\phi^*$ does not admit a closed-form expressions \cite{bayraksan2015data}.
Similar to the Kullback-Leibler divergence, the sample complexity obtained in Corollary \ref{corr_jeffreys} matches the lower bound presented in \eqref{eq_sample_lb_suplinear_max} in terms of their exponential dependence on the precision target $\epsilon$.
To the best of our knowledge, Corollary \ref{corr_jeffreys} above also presents  the first sample complexity bound for estimating $\cR(X)$ defined by the Jefferys divergence.

\begin{corollary}[Cressie-Read family \cite{duchi2021learning}]\label{corr_cressie}
For any $k \in \RR \setminus \cbr{0, 1}$,  let the divergence function be $\phi_k (x) = \frac{x^k - kx + k-1}{k(k-1)} \mathbbm{1}_{[0, \infty)}(x)$. 
Then the following holds.
\begin{itemize}
\item[(1)]  If $k > 1$, then for any $\delta \in (0,1)$, the number of samples $n$ to ensure $\abs{\cR(X) - \cR_n(X)} \leq \epsilon$ with probability $1-\delta$ can be bounded by 
\begin{align}\label{read_sup_growth_sample_ub}
\tilde{\cO} \rbr{
\rbr{\frac{B}{\epsilon}}^{ \max \cbr{2, {k}/\rbr{k-1}} }
}.
\end{align}
\item [(2)] If $k < 1$, then obtaining an $\epsilon$-accurate estimation of $\cR(X)$ can require an arbitrarily large number of samples.
\end{itemize}
\end{corollary}

\begin{proof}
From the definition of $\phi_k$, we have $g_{\phi_k}(x) = \frac{x^{k-1} - k + (k-1)/x}{k(k-1)}$, and hence 
$g_{\phi_k}(y) = \Theta (y^{1/(k-1)})$ for $y > 0$. 
The first part of the claim then follows from Corollary \ref{corr_sample_improved_var} with simple verifications. 
The second part of the claim follows  directly from Theorem \ref{info_lb_sublinear}.
\end{proof}

For $k > 1$, by instantiating the sample complexity lower bound \eqref{eq_sample_lb_suplinear_max} with $g_{\phi_k}(y) = \Theta (y^{1/(k-1)})$,
and comparing it with \eqref{read_sup_growth_sample_ub} in Corollary \ref{corr_cressie}, the sample average approximation estimator $\cR_n(X)$ obtains the optimal sample complexity bound of 
$\tilde{\cO} \rbr{
(\frac{B}{\epsilon})^{ \max \cbr{2, {k}/\rbr{k-1}} }
}$
for the Cressie-Read divergences.
In particular, when $k=2$, the corresponding divergence coincides with the Pearson $\chi^2$-divergence.
It should be noted that similar sample complexity upper and lower bounds have also been obtained in \cite{duchi2021learning}, which rely on a further simplification of the dual representation for the  Cressie-Read divergence functions, followed by a delicate analysis on the concentration of the empirical $L_p$-norm with $p > 1$.
In addition, for $k < 1$, the corresponding statistical intractability presented in Corollary \ref{corr_cressie} appears to be new in the literature for estimating $\cR(X)$.

%!TEX root = ./sample_dro.tex

%\newpage

\section{Concluding Remarks}\label{sec_appendix}

Before we conclude our discussions, we briefly discuss some potential problems of interest.
First, it could be  interesting to extend the sample complexity results to the case when decision variables are involved,  i.e., $\min_{\theta \in \Theta} \cR(X_\theta)$ with a compact $\Theta$.
Second, it appears that one can readily develop sample approximation-based methods \cite{nemirovski2009robust} for solving $\min_{\theta \in \Theta} \cR (X_\theta)$ via the truncated variant $\min_{\theta \in \Theta} \cR_L (X_\theta)$.
Third, it remains rewarding to study the asymptotic normality of the sample average estimator $\cR_n(X)$ and compare the result with the risk-neutral case. 
Lastly, one could further consider studying the sample average approximation for general coherent risk measures.

\bibliographystyle{plain}
\bibliography{references}

%\appendix 
%\input{appendix}

\end{document}